\documentclass[11pt]{amsart}
\pdfoutput=1
\usepackage{color}
\usepackage[colorlinks=true,urlcolor=blue,linkcolor=blue,citecolor=black]{hyperref}
\usepackage[utf8]{inputenc}
\usepackage{amssymb, amsfonts, amsmath, amsthm, mathabx}
\usepackage{verbatim}

\usepackage{latexsym}
\usepackage{stmaryrd}
\usepackage{shuffle}
\usepackage{tikz,tkz-tab, tikz-cd}

\newcommand{\pai}{\left(}
\newcommand{\pad}{\right)}

\usepackage[T1]{fontenc}   
\usepackage[english]{babel}

\usepackage{graphicx}

\newtheorem{teo}{Theorem}[section]
\newtheorem{thm}[teo]{Theorem}
\newtheorem{cor}[teo]{Corollary}
\newtheorem{lemma}[teo]{Lemma}
\newtheorem{prop}[teo]{Proposition}

\theoremstyle{definition}
\newtheorem{defi}[teo]{Definition}
\newtheorem{rem}[teo]{Remark}

\newtheorem{nota}[teo]{Notation}
\newtheorem{notrem}[teo]{Notation and Remark}

\newtheorem{exm}[teo]{Example}

\newcommand{\cc}{\mathbb{C}}

\newcommand{\nn}{\mathbb{N}}

\newcommand{\rr}{\mathbb{R}}

\renewcommand{\AA}{\mathcal{A}}

\newcommand{\CC}{\mathcal{C}}

\newcommand{\GG}{\mathcal{G}}

\newcommand{\NN}{\mathcal{N}}
\newcommand{\OO}{\mathcal{O}}
\newcommand{\PP}{\mathcal{P}}

\newcommand{\XX}{\mathcal{X}}
\newcommand{\YY}{\mathcal{Y}}

\newcommand{\lan}{\left\langle}
\newcommand{\ran}{\right\rangle}

\newcommand{\NCac}{\XX }
\newcommand{\NCacd}{\YY }
\newcommand{\NCpp}{\NN\CC^{par} }
\newcommand{\NCeven}{\NN\CC^{even} }

\usepackage{calc}
\newcounter{PartitionDepth}
\newcounter{PartitionLength}

\newcommand{\upparti}[2]{
 \begin{picture}(#2,#1)
 \setcounter{PartitionDepth}{#1}
 \put(#2,0){\line(0, 1){1}}
 \end{picture}}
\newcommand{\uppartii}[3]{
 \begin{picture}(#3,#1)
 \setcounter{PartitionLength}{#3-#2}
 \setcounter{PartitionDepth}{#1}
 \put(#2,0){\line(0,1){#1}}     
 \put(#3,0){\line(0,1){#1}}
 \put(#2,\thePartitionDepth){\line(1,0){\thePartitionLength}}
 \end{picture}}
\newcommand{\uppartiii}[4]{
 \begin{picture}(#4,#1)
 \setcounter{PartitionLength}{#4-#2}
 \setcounter{PartitionDepth}{#1}
 \put(#2,0){\line(0,1){#1}}
 \put(#3,0){\line(0,1){#1}}
 \put(#4,0){\line(0,1){#1}}
 \put(#2,\thePartitionDepth){\line(1,0){\thePartitionLength}} 
 \end{picture}}

\usepackage{forest}
\forestset{
  decor/.style = {
    label/.expanded = {[inner sep = 0.2ex, font=\unexpanded{\tiny}]right:{$#1$}}
  },
  root/.style = {minimum size = 0.1ex},
  decorated/.style = {
    for tree = {
      circle, fill, inner sep = 0.3ex, minimum size = 1.ex,
      grow' = south, l = 0, l sep = 1.2ex, s sep = 0.7em,
      fit = tight, parent anchor = center, child anchor = center,
      delay = {decor/.option = content, content =}
    }
  },
  default preamble = {decorated, root},
}

\DeclareRobustCommand{\SkipTocEntry}[5]{}

\title{On the anti-commutator of two free random variables}

\author{Daniel Perales}
\address{Daniel Perales: Department of Pure Mathematics, University of Waterloo, Ontario, Canada.}
\email{dperales@uwaterloo.ca}
\thanks{Supported by CONACyT (Mexico) via the scholarship 714236}

\begin{document}

\begin{abstract}
    Let $(\kappa_n(a))_{n\geq 1}$ denote the sequence of free cumulants of a random variable $a$ in a non-commutative probability space $(\mathcal{A},\varphi)$. Based on some considerations on bipartite graphs, we provide a formula to compute the cumulants $(\kappa_n(ab+ba))_{n\geq 1}$ in terms of $(\kappa_n(a))_{n\geq 1}$ and $(\kappa_n(b))_{n\geq 1}$, where $a$ and $b$ are freely independent. Our formula expresses the $n$-th free cumulant of $ab+ba$ as a sum indexed by partitions in the set $\mathcal{Y}_{2n}$ of non-crossing partitions of the form
  \[
  \sigma=\{B_1,B_3,\dots, B_{2n-1},E_1,\dots,E_r\}, \quad \text{with }r\geq 0,
  \]
  such that $i\in B_{i}$ for $i=1,3,\dots,2n-1$ and $|E_j|$ even for $j\leq r$. Therefore, by studying the sets $\mathcal{Y}_{2n}$ we obtain new results regarding the distribution of $ab+ba$. For instance, the size $|\mathcal{Y}_{2n}|$ is closely related to the case when $a,b$ are free Poisson random variables of parameter 1. Our formula can also be expressed in terms of cacti graphs. This graph theoretic approach suggests a natural generalization that allows us to study quadratic forms in $k$ free random variables.
\end{abstract}

\maketitle
\tableofcontents

\section{Introduction}

\subsection{The problem of the anti-commutator}

$\ $

\noindent Free cumulants, $(\kappa_n)_{n\geq 1}$, were introduced by Speicher \cite{Spe} as a combinatorial tool to study the addition of free random variables. Free cumulants can also be used to study the multiplication of free random variables (see \cite{nica1996multiplication}), via the following formula
\begin{equation}
\label{eq.cumulant.product}
    \kappa_n(ab) = \sum_{\tau\in \NN\CC(n) } \Bigg( \prod_{V\in \tau} \kappa_{|V|}(a)  \prod_{W\in Kr(\tau)} \kappa_{|W|}(b) \Bigg),
\end{equation}
where $\NN\CC(n)$ is the set of non-crossing partitions of $[n]:=\{1,\dots,n\}$ and $Kr(\tau)$ is the Kreweras complement of $\tau$ (see Definition \ref{defi.Kreweras} below).

It is then natural that one studies some other polynomials in two free random variables $a$ and $b$. Due to the non-commutative nature of $a$ and $b$, it makes sense that the next in line should be the commutator $i(ab-ba)$ and the anti-commutator $ab+ba$. It turns out that the free cumulants are also useful to study these cases. The combinatorial study of $i(ab-ba)$ was done by Nica and Speicher \cite{nica1998commutators}. The main observation is that the minus sign leads to the massive cancellation of terms in the expansions of the free cumulants of $i(ab-ba)$, leading to a tractable formula.  On the other hand for $ab + ba$ there are no cancellations to be followed, and the combinatorial approach becomes much more involved. The only exception is when $a$ and $b$ have symmetric distributions, that is, all their odd moments (and therefore all their odd cumulants) are zero. In this special case, $ab + ba$ and $i(ab-ba)$ have the same distribution and the common distribution can be computed using $R$-diagonal elements, see \cite{nica1998commutators} or Lecture 15 of \cite{NS}.

More recently, a formula to compute the Boolean cumulants of the free anti-commutator $ab+ba$ in terms of the individual Boolean cumulants of $a$ and $b$ was provided in \cite{fevrier2020using}. On the other hand, Ejsmont and Lehner studied quadratic forms of even variables \cite{ejsmont2020sums} (see also \cite{ejsmont2017sample}) and the limiting distribution that arises from sums of commutators and anti-commutators \cite{ejsmont2020free}. 
There also exist plainly analytic approaches to compute the anti-commutator. In \cite{vasilchuk2003asymptotic}, various systems of equations are given, which can in principle be used to calculate the Cauchy transform of $ab+ba$, while \cite{helton2018applications} provides an algorithm which can be used to obtain the distribution of a free anti-commutator, by using the linearization trick.

\subsection{A new formula for the anti-commutator}

$\ $

\noindent In this paper, we provide a general formula to compute the free cumulants of the anti-commutator $ab+ba$ in terms of the free cumulants of $a$ and $b$. The key idea is to carefully examine certain graphs associated to non-crossing partitions. These graphs have already appeared in \cite{mingo2012sharp}, see also Section 4.4 of \cite{mingo2017free}. In the present paper, the relevant feature we seek in these graphs is whether they are connected and bipartite.

\begin{defi}
\label{defi.Graph.pi}
\label{main.defi}
Let $\NN\CC(2n)$ be the set of non-crossing partitions of $[2n]:=\{1,2,\dots,2n\}$. 
\begin{enumerate}
\item[1.] Given a partition $\pi\in \NN\CC(2n)$, we construct the graph $\GG_\pi$ as follows:
\begin{itemize}
    \item The vertices of $\GG_\pi$ are going to be the blocks of $\pi$.
    \item For $k=1,2,\dots,n$ we draw an undirected edge between the block containing element $2k-1$ and the block containing element $2k$. We allow for loops and multiple edges, thus $\GG_\pi$ has exactly $n$ edges.
\end{itemize}

\item[2.] We denote
\[
\NCac_{2n}:=\{\pi\in \NN\CC(2n) : \GG_\pi\mbox{ is connected and bipartite}\}.
\]
A partition $\pi \in \NCac_{2n}$ has a natural \emph{bipartite decomposition} $\pi = \pi ' \sqcup \pi ''$.  Denoting by $V_1$ the block  of $\pi$ which contains the number $1$, we have that $\pi '$ consists of the blocks of $\pi$ which are at even distance from $V_1$ in the graph $\mathcal{G}_{\pi}$, while $\pi''$ consists of the blocks of $\pi$ which are at odd distance from $V_1$ in that graph.

\item[3.] For every $n\in\nn$, we denote by  $\NCacd_{2n}$ the set of non-crossing partitions of the form 
\[
\sigma=\{B_1,B_3,\dots, B_{2n-1},E_1,\dots,E_r\}, \qquad \text{with }r\geq 0,
\]
such that $i\in B_{i}$ for $i=1,3,\dots,2n-1$ and $|E_j|$ even for $j\leq r$.

In plain words, $\sigma\in \NCacd_{2n}$ if it separates odd elements and the blocks of $\sigma$ containing just even elements have even size. The notation is set to remind us that $B_i$ is the \emph{block containing $i$}, while $E_j$ is of \emph{even size} and also happens to contain only even elements, as $\{1,3,\dots,2n-1\}\subset B_1\cup\dots\cup B_{2n-1}$.
\end{enumerate}
\end{defi}

In Section \ref{sec:NCacd} we will show that the sets $\NCac_{2n}$ and $\NCacd_{2n}$ are actually related via the Kreweras complementation map:

\begin{prop}
\label{prop.NCacd.nice.description}
For every $n\in\nn$, we have that
\[
\NCac_{2n}=Kr(\NCacd_{2n}),
\]
where $Kr:\NN\CC(n)\to \NN\CC(n)$ is the Kreweras complementation map. 
\end{prop}

\begin{exm}
\label{exm.NCacd}
\begin{enumerate}
    \item[1.] Here is an example of a partition $\sigma\in\NCacd_{12}$, on the left. In the middle we have $\pi=Kr(\sigma)\in \NCac_{12}$, and on the right we have the graph $\GG_\pi$.
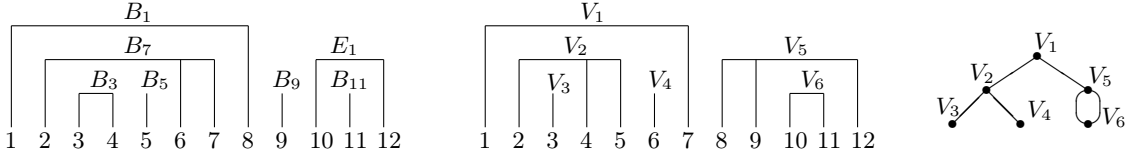
\[
\begin{footnotesize}
\begin{picture}(14,5)
\put(1,.3) 1
\put(2,.3) 2
\put(3,.3) 3
\put(4,.3) 4
\put(5,.3) 5
\put(6,.3) 6
\put(7,.3) 7
\put(8,.3) 8
\put(9,.3) 9
\put(10,.3) {10}
\put(11,.3) {11}
\put(12,.3) {12}
\put(4.5,4.1) {$B_1$}
\put(4.5,3.1) {$B_7$}
\put(8.9,2.1) {$B_9$}
\put(10.6,2.1) {$B_{11}$}
\put(10.6,3.1) {$E_{1}$}
\put(3.5,2.1) {$B_3$}
\put(5,2.1) {$B_5$}
\put(-0.05,0.9){\uppartii{3}{1}{8}}
\put(-0.05,0.9){\uppartii{1}{3}{4}}
\put(-0.05,0.9){\uppartiii{2}{2}{6}{7}}
\put(-0.05,0.9){\upparti{1}{5}}
\put(-0.05,0.9){\upparti{1}{9}}
\put(-0.05,0.9){\upparti{1}{11}}
\put(-0.05,0.9){\uppartii{2}{10}{12}}
\end{picture}
\begin{picture}(14,4)
\put(1,.3) 1
\put(2,.3) 2
\put(3,.3) 3
\put(4,.3) 4
\put(5,.3) 5
\put(6,.3) 6
\put(7,.3) 7
\put(8,.3) 8
\put(9,.3) 9
\put(10,.3) {10}
\put(11,.3) {11}
\put(12,.3) {12}
\put(4,4.1) {$V_1$}
\put(10,3.1) {$V_5$}
\put(6,2.1) {$V_4$}
\put(10.5,2.1) {$V_6$}
\put(3.5,3.1) {$V_2$}
\put(3,2) {$V_3$}
\put(-0.05,0.9){\uppartii{3}{1}{7}}
\put(-0.05,0.9){\upparti{1}{3}}
\put(-0.05,0.9){\upparti{1}{6}}
\put(-0.05,0.9){\uppartiii{2}{2}{4}{5}}
\put(-0.05,0.9){\uppartiii{2}{8}{9}{12}}
\put(-0.05,0.9){\uppartii{1}{10}{11}}
\end{picture}
\begin{picture}(5,3)
 \put(.5,1.3) {$V_3$}
 \put(1,1) {\circle*{0.2}}
 \put(1.5,2.3) {$V_2$}
 \put(2,2) {\circle*{0.2}}
 \put(3.4,3.2) {$V_1$}
 \put(3.5,3) {\circle*{0.2}}
 \put(3.2,1.2) {$V_4$}
 \put(3,1) {\circle*{0.2}}
 \put(5,2.2) {$V_5$}
 \put(5,2) {\circle*{0.2}}
 \put(5.4,1) {$V_6$}
 \put(5,1) {\circle*{0.2}}
 \put(1,1){\line(1,1){1}}
 \put(2,2){\line(3,2){1.5}}
 \put(2,2){\line(1,-1){1}}
 \put(3.5,3){\line(3,-2){1.5}}
 \put(5,1.5){\oval(.7,1)}
\end{picture}
\end{footnotesize}
\]
The bipartite decomposition of $\GG_\pi$ has $\pi'=\{V_1,V_3,V_4,V_6\}$ and $\pi''=\{V_2,V_5\}$.

$\ $

\item[2.] In the table below we list the 5 partitions $\sigma$ of $\NCacd_4$ with their corresponding $\pi:=Kr(\sigma)\in \NCac_4$ and the graph $\GG_\pi$:

\setlength{\unitlength}{0.3cm}

\begin{center}
\begin{tabular}{|c|ccccc|}
\hline
\begin{tabular}{c}  $\sigma$ \\ \vspace{.5cm} \end{tabular} 
&
\begin{picture}(5,3.5)
\put(-0.05,0.6){\upparti{1}{1}}
\put(-0.05,0.6){\upparti{1}{3}}
\put(-0.05,0.6){\uppartii{2}{2}{4}}
\end{picture}
&
\begin{picture}(5,1)
\put(-0.05,0.6){\uppartii{1}{1}{2}}
\put(-0.05,0.6){\uppartii{1}{3}{4}}
\end{picture}
&
\begin{picture}(5,1)
\put(-0.05,0.6){\upparti{1}{3}}
\put(-0.05,0.6){\uppartiii{2}{1}{2}{4}}
\end{picture}
&
\begin{picture}(5,1)
\put(-0.05,0.6){\upparti{1}{1}}
\put(-0.05,0.6){\uppartiii{1}{2}{3}{4}}
\end{picture}
&
\begin{picture}(5,1)
\put(-0.05,0.6){\uppartii{2}{1}{4}}
\put(-0.05,0.6){\uppartii{1}{2}{3}}
\end{picture}
\\
\begin{tabular}{c}  $\pi$ \\ \vspace{.5cm} \end{tabular} 
&
\begin{picture}(5,3)
\put(-0.05,0.6){\uppartii{2}{1}{4}}
\put(-0.05,0.6){\uppartii{1}{2}{3}}
\end{picture}
&
\begin{picture}(5,1)
\put(-0.05,0.6){\upparti{1}{1}}
\put(-0.05,0.6){\upparti{1}{3}}
\put(-0.05,0.6){\uppartii{2}{2}{4}}
\end{picture}
&
\begin{picture}(5,1)
\put(-0.05,0.6){\upparti{1}{1}}
\put(-0.05,0.6){\upparti{1}{4}}
\put(-0.05,0.6){\uppartii{1}{2}{3}}
\end{picture}
&
\begin{picture}(5,1)
\put(-0.05,0.6){\upparti{1}{2}}
\put(-0.05,0.6){\upparti{1}{3}}
\put(-0.05,0.6){\uppartii{2}{1}{4}}
\end{picture}
&
\begin{picture}(5,1)
\put(-0.05,0.6){\upparti{1}{2}}
\put(-0.05,0.6){\upparti{1}{4}}
\put(-0.05,0.6){\uppartii{2}{1}{3}}
\end{picture}
\\
\begin{tabular}{c}  $\GG_\pi$ \\ \vspace{.05cm} \end{tabular} 
&
\begin{picture}(2,2)
 \put(1,1.7){\circle*{0.5}}
 \put(1,.3){\circle*{0.5}}
 \put(1,1){\oval(1.2,1.5)}
\end{picture}
&
\Forest{[[][]]}
&
\Forest{[[][]]}
&
\Forest{[[][]]}
&
\Forest{[[][]]} \\  \hline 
\end{tabular}
\end{center}
\vspace{.3cm}
\end{enumerate}
\end{exm}

\setlength{\unitlength}{0.45cm}

Our main contribution is a formula that expresses the cumulants of the anti-commutator as a sum indexed by $\NCac_{2n}$, or equivalently $\NCacd_{2n}$ (as the Kreweras map is a bijection).

\begin{thm}
\label{Thm.anticommutator.main}
Consider two free random variables $a$ and $b$, and let $(\kappa_n(a))_{n\geq 1}$, $(\kappa_n(b))_{n\geq 1}$ and $(\kappa_n(ab+ba))_{n\geq 1}$ be the sequence of free cumulants of $a$, $b$ and $ab+ba$, respectively. Then, for every $n\geq 1$ one has
\begin{equation}
\label{formula.anticommutator.main}
\kappa_n(ab+ba) = \sum_{\substack{\pi\in \NCac_{2n} \\ \pi=\pi'\sqcup\pi''}} \Bigg( \prod_{V\in \pi'} \kappa_{|V|}(a)  \prod_{W\in \pi''} \kappa_{|W|}(b) + \prod_{V\in \pi'} \kappa_{|V|}(b)  \prod_{W\in \pi''} \kappa_{|W|}(a) \Bigg).
\end{equation}
\end{thm}

Equation \eqref{formula.anticommutator.main} takes a simpler form when $a$ and $b$ have the same distribution, which means that $\varphi(a^n)=\varphi(b^n)$ for all $n\in\nn$, or equivalently that $\kappa_n(a)=\kappa_n(b)$ for all $n\in\nn$. This phenomenon was also observed in \cite{fevrier2020using}, for the formula in terms of Boolean cumulants.

\begin{cor}
\label{cor.anticommutator.id}
Consider two free random variables, $a$ and $b$, with the same distribution, and let us denote $\lambda_n:=\kappa_n(a)=\kappa_n(b)$. Then, one has
\begin{equation} 
\label{eq.anticommutator.id}
\kappa_n(ab+ba) =2\cdot \sum_{\pi\in \NCac_{2n}} \prod_{V\in \pi} \lambda_{|V|} \qquad \forall n\in\nn.
\end{equation}
\end{cor}

\begin{cor}
\label{cor.anticomm.free.poisson.lambda}
If $a$ and $b$ are free random variables, both with free Poisson distribution of parameter $\lambda$, namely $\kappa_n(a)=\kappa_n(b)=\lambda$ for all $n \in \nn$, formula \eqref{eq.anticommutator.id} specializes to
\[
\kappa_n (ab+ba) = 2 \cdot \sum_{\pi \in \NCac_{2n}} \lambda^{| \pi |}
= 2 \cdot \sum_{\pi \in \NCacd_{2n}} \lambda^{ 2n+1 - | \pi |}  \qquad \forall n\in\nn. 
\]
\end{cor}

In the special case where $a,b$ are both free Poisson of parameter 1, the cumulants are just given by twice the size of $\NCacd_{2n}$. In Theorem \ref{thm.recurrence.relation} below we will obtain a recursive formula for the sizes $|\NCacd_{2}|, |\NCacd_{4}|, |\NCacd_{6}|, \dots$ and this allows us to get the compositional inverse of the moment series of the anti-commutator.

\begin{thm}
\label{thm.inverse.moment.series}
Let $\nu$ be the distribution of the free anti-commutator of two free Poisson of parameter 1. The compositional inverse of the moment series of $\nu$ is given by
\begin{equation}
  M_\nu^{\langle-1\rangle}(z)=\frac{-7z-6 + 3 \sqrt{(z+2)(9z+2) }}{4(z+2)^2(z+1)}.
\end{equation}
\end{thm}

\begin{rem}  
 To get a better grasp of Theorem \ref{Thm.anticommutator.main}, we now identify in Equation \eqref{formula.anticommutator.main} all those terms that also appear in the expansion of $\kappa_n (ab)$ from formula \eqref{eq.cumulant.product}. Let $\mathcal{NC}_2 (2n)$ be the set of all non-crossing pair-partitions of $[2n]$, that is, the set of partitions $\sigma = \{ V_1, \ldots , V_n \} \in \mathcal{NC} (2n)$ with $|V_1| = \cdots = |V_n| = 2$.  It is immediate that $\mathcal{NC}_2 (2n) \subseteq \mathcal{Y}_{2n}$; indeed, the non-crossing condition forces every pair of a $\sigma \in \mathcal{NC}_2 (2n)$ to contain one odd number and one even number, hence the pairs $V_1, \ldots , V_n$ of $\sigma$ can be renamed as $B_1, B_3,\ldots , B_{2n-1}$ with $i\in B_i$ for $i=1,3,\dots, 2n-1$.

As a consequence, it follows that
\[
\mathcal{X}_{2n} \supseteq \mathrm{Kr} ( \mathcal{NC}_2 (2n) )= \{ \lan \tau,Kr(\tau)\ran \mid \tau \in \mathcal{NC} (n) \},
\]

where for every $\tau \in \mathcal{NC} (n)$ we denote by $ \lan \tau,Kr(\tau)\ran $ the partition in $\mathcal{NC} (2n)$ obtained by placing $\tau$ on positions $\{ 1,3, \ldots , 2n-1 \}$ and its Kreweras complement on positions $\{ 2,4, \ldots , 2n \}$.  (See Section \ref{ssec:partitions} below for a precise discussion of this procedure.)

When we look at the summation over $\mathcal{X}_{2n}$ indicated on the right-hand side of Equation \eqref{formula.anticommutator.main}, and check the terms of that summation which are indexed by partitions $ \lan \tau,Kr(\tau)\ran$ with $\tau \in NC(n)$, we run into a double copy of the summation shown in \eqref{eq.cumulant.product}, giving the free cumulants $\kappa_n (ab)$ and $\kappa_n (ba)$.  
\end{rem}

\subsection{Rephrasing in terms of cacti graphs}

$\ $

\noindent 
Our main formula, \eqref{formula.anticommutator.main}, can be also phrased solely in terms of graphs. 

\begin{defi}
Let $\GG$ be a graph, where we allow loops and multiple edges.
\begin{itemize}
    \item[1.] A \emph{cycle} is a finite sequence of vertices and edges $(v_1,e_1,v_2,e_2,\dots,v_j,e_j)$ with vertices $v_1,\dots,v_j\in V_\GG$ and where $e_i\in E_\GG$ connects $v_i$ with $v_{i+1}$ for $i=1,\dots,j$ (assuming $j+1=1$). We say that the cycle is \emph{simple} if there is no proper subset $\{v_{i_1},\dots,v_{i_k}\}\subset \{v_1,\dots,v_j\}$ such that $v_{i_1},\dots,v_{i_k}$ are the vertices of another cycle. 
    \item[2.] A \emph{cactus} is a connected graph in which every edge $e\in E_\GG$ belongs to at most one simple cycle. We say that $e$ is \emph{rigid} if it belongs to a simple cycle, and we say that $e$ is \emph{flexible} if it does not belong to a simple cycle. The term cactus has already appeared in Traffic Freeness in connection with Free Probability, but a our notion is slightly different, see Remark \ref{rem.cactus.traffics} for details.
    \item[3.] An \emph{outercycle (or orientation)} of a cactus graph $\GG$ is a cycle $C$ that passes exactly once through every rigid edge and twice through every flexible edge. Every cactus graph has an outercycle, see  Remark \ref{rem.outercycle} below. An \emph{oriented cactus graph} is a pair $(\GG,C)$ where $\GG$ is a cactus graph and $C$ is an outercycle. We denote by $OCG_n$ the set of oriented cacti graphs with exactly $n$ edges.
\end{itemize}
Note: two orientations $C=(v_1,e_1,v_2,e_2,\dots,v_j,e_j)$ and $C'=(v'_1,e'_1,v'_2,e'_2,\dots,v'_j,e'_j)$ of a cactus graph $\GG$ are considered to be the same if there exists an automorphism $g$ of the graph $\GG$ such that $g(v_i)=v'_i$ and $g(e_i)=e'_i$ for $i=1,\dots,j$.
\end{defi}

Now let us see what kind of graph $\GG_\pi$ can be.

\begin{prop}
\label{prop.cactus}
Let $\pi\in \NN\CC(2n)$ be such that $\GG_\pi$ is connected. Then, $\GG_\pi$ is a cactus graph. 
\end{prop}

In particular, the set $\{ \GG_\pi | \pi\in \NCac_{2n}\}$ consists of cacti graphs. We next want to study the size of $\{ \pi\in \NCac_{2n}: \GG_\pi=\GG\}$ for a given cactus graph $\GG$. This can be done effectively, once we have an outercycle.

\begin{notrem}
Consider a partition $\pi\in\NN\CC(2n)$ such that $\GG_\pi$ is connected. It is possible to construct a canonical outercycle $C_\pi$ of the cactus $\GG_\pi$. This outercycle starts in the block $V_1$ of $\pi$ that contains the number 1 and `goes around' $\GG_\pi$, it is concisely defined as follows. Let $S_\pi$ be the permutation of the symmetric group in $2n$ elements whose orbits are the blocks of $\pi$, where the elements of each block are run in increasing order. Namely, if $V=\{y_1<y_2<\dots<y_r\}$ is a block of $\pi$ then $S_\pi(y_i)=y_{i+1}$ for $i=1,\dots,r$ (considering $r+1=1$). Also, let $\tau:=(12)(34)\dots(2n-1, 2n)$ be the permutation that transposes pairs of consecutive elements. Then consider the orbit of 1 in $S_\pi\circ\tau$, and denote it as $x_1,x_2,\dots, x_j$ with $x_1=1$. If we let $v_i$ be the block containing the element $x_i$, and $e_i$ be the edge joining the blocks containing elements $x_i$ and $\tau(x_i)$, then the outercycle is given by $C_\pi=(v_1,e_1,v_2,e_2,\dots,v_j,e_j)$.
\end{notrem}

\begin{exm}
\label{exm.cactus}
We now explain how to obtain the canonical outercycle $C_\pi$ of the graph $\GG_\pi$ associated to the partition $\pi$ from the first part of Example \ref{exm.NCacd}: 
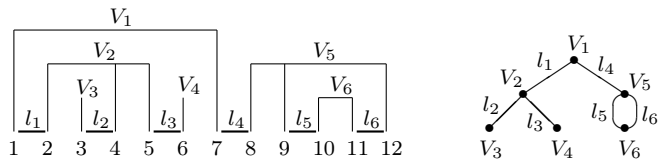
\begin{center}
\begin{scriptsize}
\begin{picture}(14,4)
\put(1,.2) 1
\put(2,.2) 2
\put(3,.2) 3
\put(4,.2) 4
\put(5,.2) 5
\put(6,.2) 6
\put(7,.2) 7
\put(8,.2) 8
\put(9,.2) 9
\put(10,.2) {10}
\put(11,.2) {11}
\put(12,.2) {12}
\put(4,4.1) {$V_1$}
\put(10,3.1) {$V_5$}
\put(6,2.1) {$V_4$}
\put(10.5,2.1) {$V_6$}
\put(3.5,3.1) {$V_2$}
\put(3,2) {$V_3$}
\put(-0.05,0.9){\uppartii{3}{1}{7}}
\put(-0.05,0.9){\upparti{1}{3}}
\put(-0.05,0.9){\upparti{1}{6}}
\put(-0.05,0.9){\uppartiii{2}{2}{4}{5}}
\put(-0.05,0.9){\uppartiii{2}{8}{9}{12}}
\put(-0.05,0.9){\uppartii{1}{10}{11}}
\thicklines
\put(1.3,.9){\line(1,0){.8}}
\put(3.3,.9){\line(1,0){.8}}
\put(5.3,.9){\line(1,0){.8}}
\put(7.3,.9){\line(1,0){.8}}
\put(9.3,.9){\line(1,0){.8}}
\put(11.3,.9){\line(1,0){.8}}
 \put(1.5,1.1) {$l_1$}
 \put(3.5,1.1) {$l_2$}
 \put(5.5,1.1) {$l_3$}
 \put(7.5,1.1) {$l_4$}
 \put(9.5,1.1) {$l_5$}
 \put(11.5,1.1) {$l_6$}
\end{picture}
\begin{picture}(5,5)
 \put(.7,.2) {$V_3$}
 \put(1,1) {\circle*{0.2}}
 \put(1.3,2.3) {$V_2$}
 \put(2,2) {\circle*{0.2}}
 \put(3.3,3.3) {$V_1$}
 \put(3.5,3) {\circle*{0.2}}
 \put(2.8,.2) {$V_4$}
 \put(3,1) {\circle*{0.2}}
 \put(5.1,2.2) {$V_5$}
 \put(5,2) {\circle*{0.2}}
 \put(4.8,.2) {$V_6$}
 \put(5,1) {\circle*{0.2}}
 \put(2.4,2.8) {$l_1$}
 \put(.8,1.5) {$l_2$}
 \put(2.1,1) {$l_3$}
 \put(4.2,2.7) {$l_4$}
 \put(4,1.3) {$l_5$}
 \put(5.5,1.2) {$l_6$}
 \put(1,1){\line(1,1){1}}
 \put(2,2){\line(3,2){1.5}}
 \put(2,2){\line(1,-1){1}}
 \put(3.5,3){\line(3,-2){1.5}}
 \put(5,1.5){\oval(.7,1)}
\end{picture}
\end{scriptsize}
\end{center}
To the left we have the partition $\pi$ with 6 extra segments $l_i$ that remind us of the pairs of consecutive elements $2i-1$ and $2i$ that will form the edges of the graph. Notice that $\GG_\pi$, which is depicted to the right, is indeed a cactus graph. It has four flexible edges $l_1,l_2,l_3,l_4$, and two rigid edges $l_5,l_6$ coming from a multiple edge that forms a simple cycle of size 2. For this partition $\pi$, we have that $S_\pi=(1,7)(2,4,5)(3)(6)(8,9,12)(10,11)$ and if we alternately apply $\tau$ and $S_\pi$ to 1 we get
\begin{equation}
\label{eq.orbit}
    1,2,\ 4,3,\ 3,4,\ 5,6,\ 6,5,\ 2,1,\ 7,8,\ 9,10,\ 11,12,\ 8,7,\ 1,2, \dots
\end{equation}
Therefore, the orbit of 1 in $S_\pi\circ\tau$ is $1,4,3,5, 6,2,7,9,11,8$, and the canonical outercycle is
\[
C_\pi=(V_1,l_1,V_2,l_2,V_3,l_2,V_2,l_3,V_4,l_3,V_2,l_1,V_4,l_4,V_5,l_5,V_6,l_6,V_5,l_4).
\]

There is a simple way to obtain $C_\pi$ from the drawing of $\pi$ together with the six segments. Imagine that the lines of the diagram represent the walls of a house seen from above, and that we are standing between 1 and 2, touching the wall $l_1$ from the south (from below) with the left hand. Then we start walking around the house always touching the walls with the left hand. If we record in order the blocks $V_i$ and lines $l_j$ that we pass when going around, we end up with $C_\pi$. Moreover, if we record the numbers we encounter while going around we get the sequence \eqref{eq.orbit}. The first 1 and 2 (or $l_1$) we encounter is our starting point (to the south of the wall). The second time we reach $l_1$, we are on the other side of the wall (to the north) of our starting point, the last time we reach 1 and 2 we are back in our starting point. The reason why we did not get $10$ and $12$ in the orbit, is because there is a `secret' chamber enclosed by $V_5$, $V_6$, $l_5$ and $l_6$. This chamber corresponds to the unique simple cycle of $\GG_\pi$. 
Notice that in the planar drawing of $\GG_\pi$ this amounts for going around the outer (unbounded) face in counterclockwise direction. So $C_\pi$ is indeed a cycle that surrounds $\GG_\pi$. Furthermore, notice $C_\pi$ passes once through every rigid edge and twice through every flexible edge. We can get other outercycles of $\GG_\pi$ by cyclically permuting the entries of $C_\pi$, but these are not considered the canonical cycle corresponding to $\pi$.
\end{exm}

We now present a concise formula for the size of the preimage of an oriented cactus graph.

\begin{prop}
\label{prop.preimage}
Let $(\GG,C)\in OCG_n$ with $C=(v_1,e_1,v_2,\dots,v_j,e_j)$ and denote by $f_C$ the number of flexible edges in $(\GG,C)$ distinct from $e_1$. Namely, if $e_1$ is rigid, then $f_C$ is simply the number of flexible edges in $\GG$ and if $e_1$ is flexible, $f_C$ is the number of flexible edges of $\GG$ minus 1. Then, 
\[
|\{\pi\in \NN\CC(2n): (\GG_\pi,C_\pi)=(\GG,C)\}|=2^{f_C}.
\]
In other words, there are exactly $2^{f_C}$ partitions $\pi\in \NN\CC(2n)$ such that $\GG_\pi=\GG$ and whose canonical cycle is $C$. 
\end{prop}

The previous results allow us to rewrite Theorem \ref{Thm.anticommutator.main} just in terms of graphs. 
\begin{thm}
\label{thm.anticommutator.graphs}
Consider two free random variables $a$ and $b$, and let $(\kappa_n(a))_{n\geq 1}$, $(\kappa_n(b))_{n\geq 1}$ and $(\kappa_n(ab+ba))_{n\geq 1}$ be the sequence of free cumulants of $a$, $b$ and $ab+ba$, respectively. Then, for every $n\geq 1$ one has:
\begin{equation}
\label{eq.anticommutator.graphs}
\kappa_n(ab+ba) = \sum_{\substack{(\GG,C)\in OCG_{n}\\ \GG\text{ is bipartite}}} 2^{f_C} \Bigg( \prod_{v\in V'_\GG } \kappa_{d(v)}(a)  \prod_{w\in V''_\GG} \kappa_{d(w)}(b) +\prod_{v\in V'_\GG }\kappa_{d(v)}(b)  \prod_{w\in V''_\GG} \kappa_{d(w)}(a) \Bigg),
\end{equation}
where $f_C$ is the number of flexible edges in $(\GG,C)$ without counting the first edge of $C$; $(V'_\GG, V''_\GG)$ is the unique bipartition of the vertices $V_\GG$ such that $V'_\GG$ contains the first vertex of $C$; and $d(v)$ is the degree of the vertex $v$ in $\GG$.
\end{thm}

Again, if $a$ and $b$ have the same distribution the formula simplifies. 

\begin{cor}
\label{cor.anticommutator.graphs.id}
Consider two free random variables, $a$ and $b$, with the same distribution, and let us denote $\lambda_n:=\kappa_n(a)=\kappa_n(b)$. Then, one has
\begin{equation} 
\label{eq.anticommutator.graph.id}
\kappa_n(ab+ba) =\sum_{\substack{(\GG,C)\in OCG_{n}\\ \GG\text{ is bipartite}}} 2^{f_C+1} \prod_{v\in V_\GG } \lambda_{d(v)}, \qquad \forall n\in\nn.
\end{equation}
\end{cor}

Another application of Theorem \ref{thm.anticommutator.graphs} concerns the free anti-commutator with a semicircular element. 

\begin{prop}
\label{prop.anticom.semicircular.arbitrary}
Consider two free random variables $a$ and $s$, where $s$ is semicircular. This means that $\kappa_2(s)=1$ and $\kappa_n(s)=0$ for every $n\neq 2$, and let us denote $\lambda_n:=\kappa_n(a)$. Then, for odd $m\in\nn$ one has $\kappa_m(as+sa)=0$, while for even $m:=2n$, one has
\begin{equation}
\label{eq.anticomm.semicircular.arbitrary}
\kappa_{2n}(as+sa) = \sum_{\substack{(\GG,C)\in OCG_{n} }} 2^{g_{C}+1} \prod_{v\in V_{\GG} } \lambda_{d(v)},
\end{equation}
where if the first edge, $e_1$, of $C$ is flexible then $g_C:=2f_C+1$, while if the first edge of $C$ is rigid then $g_C:=2f_C$. 
\end{prop}

Notice that Equation \eqref{eq.anticommutator.graph.id} is similar to Equation \eqref{eq.anticomm.semicircular.arbitrary}, except that the former is restricted to bipartite graphs and the exponents of 2 differ. Compare this to the analogue equations where we pick cumulants of products instead of anti-commutators. Namely, if $a,b,s$ are free random variables, with $a,b$ having the same distribution and $s$ being semicircular, then it follows from Equation \eqref{eq.cumulant.product} that $\kappa_{m}(as)=0$ for $m$ odd and $\kappa_{2n}(as)=\kappa_n(ab)$ for every $n\in\nn$. Thus, for the cumulants of products these two formulas coincide.

\subsection{Generalization to quadratic forms}

$\ $

\noindent 
The graph approach suggests a very natural generalization which is useful to study quadratic forms:
\[
a:= \sum_{1\leq i,j\leq k} w_{i,j} a_ia_j,
\]
where $a_1,\dots, a_k$ are free random variables, $w_{ij}\in \rr$ and we require $w_{j,i}=w_{i,j}$ for $1\leq i\leq j\leq k$, so that $a$ is self-adjoint. Even in this generality we just make use of cacti graphs, the difference is that instead of bipartite graphs we need to consider $k$-colored graphs.

\begin{nota}
We will denote by $OCG^{(k)}_n$, the set of $k$-colored oriented cacti graphs with $n$ edges. Namely, we take an oriented cactus $(\GG,C)\in OCG_n$ together with a coloring $\lambda:V_\GG\to[k]$ of the vertices of $\GG$. For $i=1,\dots,k$, we denote by $Q_i=:\lambda^{-1}(i)\subset V_\GG$ the subset of vertices that has color $i$. Given a sequence of weights $(w_{i,j})_{1\leq i,j\leq k}$, and a graph $\GG\in OCG^{(k)}_n$ we denote by
\[
w_\GG:=\prod_{(v,u)\in E_\GG} w_{\lambda(v),\lambda(u)}
\]
the product of the weights of each edge of the graph $\GG$.
\end{nota}

\begin{thm}
\label{thm.quadratic.forms}
Consider free random variables $a_1,\dots, a_k$ and let
\[
a:= \sum_{1\leq i,j\leq k} w_{i,j} a_ia_j,
\]
where $w_{i,j}\in \rr$ and $w_{j,i}=w_{i,j}$ for $1\leq i\leq j\leq k$. Then, the cumulants of $a$ are given by 
\begin{equation}
\label{eq.quadratic.forms}
\kappa_n(a) = \sum_{(\GG,C) \in OCG^{(k)}_{n}} 2^{f_C} w_{\GG} \prod_{i=1}^k \Big( \prod_{v_i\in Q_i } \kappa_{|v_i|}(a_i) \Big).
\end{equation}
\end{thm}

If we restrict our attention to the case where all the variables are even, then we can retrieve Proposition 4.5 of \cite{ejsmont2017sample} (see also \cite{ejsmont2020sums}). The details are given in Corollary \ref{cor.quadratic.forms.even}. \\

Besides this introductory section, this paper has five other sections. Section \ref{sec:prelim} reviews the basic notions on non-crossing partitions, free cumulants and free probability, needed later on in the paper. In Section \ref{sec:anticomm} we prove our main Theorem \ref{Thm.anticommutator.main} describing the cumulants of the anti-commutator. In Section \ref{sec:NCacd}, we prove Proposition \ref{prop.NCacd.nice.description} and study in more detail the set $\NCacd_{2n}$. These results are used for the applications in Section \ref{sec:other.applications}, which is divided in two parts. First we study the case of two free Poisson of parameter 1 and prove Theorem \ref{thm.inverse.moment.series}, then we use our method to retrieve the previously known formula in the case where distributions are assumed to be symmetric. Finally, in Section \ref{sec:quadratic.forms} we explore a graph theoretic approach and prove Theorem \ref{thm.anticommutator.graphs} and Theorem \ref{thm.quadratic.forms}.\\

Acknowledgements:  The author expresses his gratitude to Alexandru Nica for his valuable feedback and continuous guiding during the elaboration of this paper. 

\section{Preliminaries}
\label{sec:prelim}

\addtocontents{toc}{\SkipTocEntry}
\subsection{Partitions}
\label{ssec:partitions}

$\ $

\noindent 
We begin with a brief introduction to the definitions and notations used in this paper regarding partitions of a set. For detailed discussion of all of these concepts the standard reference is \cite{NS}.  A \textit{partition $\pi$ of a finite set $S$} is a set of the form $\pi=\{V_1,\dots,V_k\}$ where $V_1,\dots, V_k\subset S$ are pairwise disjoint non-empty subsets of $S$ such that $V_1\cup \dots \cup V_k=S$. The subsets $V_1,\dots, V_k$ are called \textit{blocks} of $\pi$, and we write $|\pi|$ for number of blocks (in this case $k$).  We denote by $\PP(S)$ the set of all partitions of $S$, and we further write $\PP(n)$ in the special case where $S=[n]:=\{1,\dots,n\}$. 

Given a partition $\pi\in\PP(n)$, we say that $V\in\pi$ is an \emph{interval block} if it is of the form $V=\{i,i+1,\dots,i+j\}$ for some integers $1\leq i<i+j \leq n$. We further say that $\pi\in\PP(n)$ is an \textit{interval partition} if all the blocks $V\in \pi$ are intervals.

\begin{defi} 
We say that $\pi\in\PP(n)$ is a \textit{non-crossing partition} if for every $1\leq i < j < k < l \leq n$ such that $i, k$ belong to the same block $V$ of $\pi$ and $j,l$ belong to the same block $W$ of $\pi$, then it necessarily follows that all $i,j,k,l$ are in the same block, namely $V=W$. We will denote by $\NN\CC(n)$ the set of all non-crossing partitions of $[n]$.
\end{defi}

Given a totally ordered set $S$, we can define the notion of interval and non-crossing partition of $S$, by taking the (unique) order preserving bijection from $S$ to $\{1,\dots,|S|\}$, and using the given definition for $n=|S|$. Through this note we will just use the special case where $S\subset [n]$ and the total order for $S$ is inherited from the total order of $[n]$. We will denote by $\NN\CC(S)$ the set of all non-crossing partitions of $S$.

\begin{nota}[Restriction]
\label{nota.restriction}
Given a partition $\pi\in \PP(n)$ and $S\subset[n]$ we denote by $\pi|S$ the partition of $\PP(|S|)$ obtained by first considering the partition in $\PP(S)$ with blocks $\{V_1\cap S,V_2\cap S,\dots, V_k\cap S\}$ and then re-indexing the elements of $S$ with $1,2,\dots, |S|$ to get a partition of $\PP(|S|)$. Notice that restrictions of non-crossing (resp. interval) partitions are again non-crossing (resp. interval) partitions.
\end{nota}

We can equip $\PP(n)$ with a lattice structure using the reverse refinement order $\leq$. For $\pi,\sigma\in \PP(n)$, we write ``$\pi \leq \sigma$'' if every block of $\sigma$ is a union of blocks of $\pi$. The maximal element of $\PP(n)$ with this order is $1_n:=\{\{1,\dots,n\}\}$ (the partition of $[n]$ with only one block), and the minimal element is $0_n:=\{\{1\},\{2\},\dots,\{n\}\}$ (the partition of $[n]$ with $n$ blocks). With this order, $(\NN\CC(n),\leq)$ is a sub-poset of $(\PP(n),\leq)$.

\begin{defi}
\label{defi.joins}
Given two partitions $\pi,\sigma\in \PP(n)$, the \emph{join of $\pi$ and $\sigma$ in the lattice of all partitions}, denoted as $\pi\lor_\PP \sigma$, is defined as the smallest partition that is bigger than $\pi$ and $\sigma$ in the inverse refinement order. 

Given two partitions $\pi,\sigma\in \NN\CC(n)$, the \emph{join of $\pi$ and $\sigma$ in the lattice of non-crossing partitions}, denoted simply as $\pi\lor \sigma$, is defined as the smallest non-crossing partition that is bigger than $\pi$ and $\sigma$ in the reverse refinement order. 
\end{defi}

Notice that definition of the join depends on which lattice we are working with. We are only concerned with the lattice $\NN\CC(n)$, but we will work with special a case where both joins coincide, and in this case we can make use of a explicit description of $\pi\lor_\PP\sigma$.

\begin{rem}[Some properties of the joins]
\label{rem.equivalence.join}
\label{rem.lor.tildelor}
If one of the partitions $\pi$ or $\sigma$ is an interval partition, then the two joins coincide $\pi\lor_\PP \sigma =\pi\lor \sigma$ (see Exercise 9.43 of \cite{NS}). Moreover, for two partitions $\pi,\sigma\in\PP(n)$, it is known that $\pi\lor_\PP \sigma =1_{n}$ if and only if for every two elements $i,j\in [n]$ we can find elements $i=i_0,i_1,i_2,\dots, i_{2k+1}=j$ such that for $l=0,\dots,k$, the pair $i_{2l},i_{2l+1}$ is in the same block of $\sigma$, and the pair $i_{2l+1},i_{2l+2}$ is in the same block of $\pi$. A detailed discussion on these and other properties of the join can be found in Lecture 9 of \cite{NS}. In this paper, we are particularly interested in partitions $\pi$ such that $\pi\lor I_{2n}=1$ for $I_{2n}:=\{\{1,2\},\{3,4\},\dots,\{2n-1,2n\}\}$. Since $I_{2n}$ is an interval partition, this means that $\pi\lor \sigma=\pi\lor_\PP I_{2n}$ and we have the nice description presented above.
\end{rem}

Throughout this paper we are going to encounter some well studied subsets of $\NN\CC(n)$. We will exploit the fact that some of these subsets are well behaved when applying the Kreweras complement map. The rest of this subsection is devoted to present and fix the notation for all these subsets and discuss the Kreweras complement map.

\begin{defi}
Given a partition $\pi\in\PP(2n)$ we say that 
\begin{itemize}
    \item $\pi$ is \emph{parity preserving} if every block $V\in\pi$ is contained either in $\{1,3,\dots,2n-1\}$ or in $\{2,4,\dots,2n\}$. We denote the set of parity preserving non-crossing partitions of $[2n]$ by $\NCpp(2n)$.
    \item $\pi$ is \emph{even} if every block $V\in\pi$ has even size. We denote the set of non-crossing even partitions of $[2n]$ by $\NCeven(2n)$.
    \item $\pi$ is \emph{a pairing (or a pair partition)} if every block $V\in\pi$ size $|V|=2$. We denote the set of non-crossing pair partitions of $[2n]$ by $\NN\CC_{2}(2n)$.
\end{itemize}
Notice that the definitions are in general for every partition (not necessarily non-crossing), but we give the notation of the set just for the non-crossing partitions as we are only concerned with these partitions.
\end{defi}

\begin{nota}
\label{nota.odd.even.partition}
Observe that given two partitions $\sigma_1,\sigma_2\in\NN\CC(n)$, there is a unique partition $\sigma\in\PP(2n)$ such that $\sigma$ is parity preserving,  $\sigma|\{1,3,\dots,2n-1\}=\sigma_1$ and $\sigma|\{2,4,\dots,2n\}=\sigma_2$, we denote this partition $\sigma$ by $\lan\sigma_1,\sigma_2\ran$.
\end{nota}

Notice that $\sigma_1,\sigma_2$ being non-crossing is a necessary but not sufficient condition for $\lan\sigma_1,\sigma_2\ran$ to be non-crossing. The Kreweras complement of a fixed partition $\sigma_1$ is actually the biggest partition $\sigma_2\in\NN\CC(n)$ such that $\lan\sigma_1,\sigma_2\ran$ is non-crossing.

\begin{defi}[Kreweras complement]
\label{defi.Kreweras}
Given $\pi\in\NN\CC(n)$, its \emph{Kreweras complement}, $Kr_n(\pi)\in\NN\CC(n)$, is defined as the biggest partition of $\NN\CC(n)$ such that $\lan \pi,Kr_n(\pi)\ran\in \NN\CC(2n)$. We will omit the sub-index $n$ in $Kr_n$ whenever it is clear from the context.
\end{defi}

\begin{rem}
\label{rem.properties.Kreweras}
Since $\lan \pi,0_n\ran$ is always non-crossing, then the set of non-crossing partitions such that $\lan \pi,Kr_n(\pi)\ran\in \NN\CC(2n)$ is not empty. Since $\NN\CC(n)$ is a lattice, there exists a unique maximum, which means that the Kreweras complement is well-defined. Moreover, the map $Kr:\NN\CC(n)\to\NN\CC(n)$ is actually a lattice anti-isomorphism of $\NN\CC(n)$. In particular, for every $\pi,\sigma\in\NN\CC(n)$ we have that $Kr(\pi\lor\sigma)=Kr(\pi)\land Kr(\sigma)$ and $Kr(\pi\land\sigma)=Kr(\pi)\lor Kr(\sigma)$. Another important property of this map is that $|Kr(\pi)|=n+1-|\pi|$ for all $\pi\in\NN\CC(n)$, thus we actually have that $Kr(1_n)=0_n$ and $Kr(0_n)=1_n$.

We are also going to use the inverse of the Kreweras complement $Kr^{-1}$, which for $\sigma\in\NN\CC(n)$ can be explicitly defined as the biggest partition $Kr^{-1}(\sigma)\in\NN\CC(n)$ such that $\lan Kr^{-1}(\sigma),\sigma\ran\in \NN\CC(2n)$. 
\end{rem}

\begin{rem}
\label{rem.even.parity.preserving}
\label{rem.Kreweras.pair.partitions}
An interesting fact is that the Kreweras complement of an even partitions is a parity preserving partition, and vice-versa. Namely we have that $Kr_{2n}(\NCeven(2n))=\NCpp(2n)$ and $Kr_{2n}(\NCpp(2n))=\NCeven(2n)$ (see Exercise 9.42 of \cite{NS}.)

Notice that $\NN\CC_2(2n)\subset \NCeven(2n)$, and this implies that 
\[
Kr_{2n}(\NN\CC_2(2n))\subset Kr_{2n}(\NCeven(2n))=\NCpp(2n).
\]
Moreover, the pair partitions are in some sense the minimal even partitions, and thus when taking the Kreweras complement we get the maximal parity preserving partitions. Specifically, the image of the set of pair partitions under the Kreweras complement map can be described as
\[
Kr_{2n}(\NN\CC_2(2n))=\{\lan \pi,Kr_{n}(\pi)\ran\in\NN\CC(2n): \pi\in \NN\CC(n)\}.
\]
Vice-versa, the Kreweras complement of partitions of this type gives again a pair partition.
\end{rem}

\addtocontents{toc}{\SkipTocEntry}
\subsection{Free cumulants}
\label{ssec:cumulants}

$\ $

\noindent 
A \textit{non-commutative probability space} is a pair $(\AA,\varphi)$, where $\AA$ is a unital algebra over $\mathbb{C}$ and $\varphi: \AA \to \cc$ is a linear functional, such that $\varphi(1_\AA) =1$. The \textit{$n$-th multivariate moment} is the multilinear functional $\varphi_n: \AA^n \to \mathbb{C}$, such that $\varphi_n(a_1, \ldots, a_n):=\varphi(a_1\cdots a_n) \in \mathbb{C}$, for elements $a_1, \ldots, a_n \in \AA$. In this framework, Voiculescu's original definition of freeness for random variables explains how to compute the mixed moments in terms the moments of each variable (see \cite{voiculescu1992free}). Since this definition is out of the scope of this paper, we will restrict our presentation to the characterization of freeness using free cumulants.

\begin{nota}\label{notat:multiplicative}
Given a family of multilinear functionals $\{f_m :\AA^{m}\to \mathbb{C}\}_{m\geq1}$ and a partition $\pi\in\PP(n)$, we define $f_\pi :\AA^{m}\to \mathbb{C}$ to be the map
\[
f_\pi(a_1, \ldots, a_n):=\prod_{V \in \pi}f_{|V|}(a_V), \qquad \forall a_1,\dots,a_n\in\AA,
\] 
where for every block $V:=\{{v_1}, \cdots, {v_k}\}$ of $\pi$ (such that ${v_1} < \cdots < {v_k}$ are in natural order) we use the notation $f_{|V|}(a_V):=f_{k}(a_{v_1},\ldots,a_{v_k})$.
\end{nota}

With this notation in hand we can define the free cumulants.

\begin{defi}[Free cumulants]
Let $(\AA,\varphi)$ be a non-commutative probability space. The \textit{free cumulants} are the family of multilinear functionals $\{\kappa_n : \AA^{n}\to \cc\}_{n\geq 1}$ recursively defined by the following formula:
\begin{equation}
\label{FreeMomCum}
	\varphi_n(a_1, \dots, a_n) =\sum_{\pi \in \NN\CC(n)} \kappa_\pi(a_1, \dots, a_n).
\end{equation} 
If we are just working with one variable, and all the arguments are the same $a_1=a_2= \dots= a_n=a$, then we adopt the simpler notation $\kappa_n(a):=\kappa_n(a, \dots, a)$ and $\kappa_\pi(a):=\kappa_\pi(a, \dots, a)$
\end{defi}

\begin{rem}
Cumulants are well defined since the right-hand side of the equation contains only one $\kappa_n$ term and the other terms are monomials of cumulants of smaller sizes. Thus we can recursively define $\kappa_n$ in terms of $\varphi_n$ and $\kappa_{n-1},\kappa_{n-2},\dots, \kappa_1$.
\end{rem}

\begin{thm}[Vanishing of mixed cumulants, see e.g.   \cite{NS}, Lecture 11]
Given a non-commutative probability space $(\AA,\varphi)$ and $a,b\in\AA$ two random variables. Then the following two statements are equivalent: 
\begin{enumerate}
    \item[1.] $a$ and $b$ are free.
    \item[2.] Every mixed cumulant vanishes. Namely, for every $n\geq 2$ and $a_1,\dots,a_n\in\{a,b\}$ which are not all equal, we have that $\kappa_n(a_1,\dots,a_n)=0$.
\end{enumerate}
\end{thm}

Finally, when working with cumulants whose entries have products of the underlying algebra $\AA$, there is an efficient formula that allows us to write this cumulant as a sum over cumulants with more entries, where the products are now separated into different entries. The general formula was found in \cite{krawczyk2000combinatorics} and is known as the products as arguments formula. Here we will just use a particular case.

\begin{thm}[Products as arguments formula]
\label{thm.products.arguments}
Let $(\AA,\varphi)$ be a non-commutative probability space and fix $n\in\nn$. Let $a_1,a_2,\dots ,a_{2n}\in\AA$ be random variables, and consider $I_{2n}:=\{\{1,2\},\{3,4\},\dots,\{2n-1,2n\}$ the unique interval pair partition. Then we have that
\begin{equation}
\label{eq.products.arguments}
    \kappa_{n}(a_1a_2,a_3a_4,\dots,a_{2n-1}a_{2n})=\sum_{\substack{\pi\in\NN\CC(2n)\\ \pi \lor I_{2n}=1_{2n} }} \kappa_{2n}(a_1,a_2,a_3,a_4,\dots,a_{2n-1},a_{2n}).
\end{equation}
\end{thm}

\section{A formula for the anti-commutator}
\label{sec:anticomm} 

The goal of this section is to prove our main result, Theorem \ref{Thm.anticommutator.main}. The approach relies on first performing standard computations using linearity of the cumulants and products as arguments formula, to get a sum indexed by partitions $\pi$, this is the content of Proposition \ref{prop.basic.anticomm}. Then, the key idea is that the graph $\GG_\pi$ puts into evidence many partitions $\pi$ that do not really contribute to the previous sum, this is done in Proposition \ref{prop.few.nonvanishing.epsilons}. This ultimately yields that the sum is indexed just by $\NCac_{2n}$. The proof of Proposition \ref{prop.NCacd.nice.description}, which is independent from Theorem \ref{Thm.anticommutator.main} will be given in Section \ref{sec:NCacd}. 

Through this section we are going to fix two free random variables $a,b$, and a natural number $n\in\nn$. Our ultimate goal is to describe the $n$-th free cumulant of the anti-commutator $\kappa_n(ab+ba)$ in terms of the cumulants of $a$ and $b$. By multilinearity of the cumulants this amounts to study $n$-th cumulants with entries given by $ab$ or $ba$. To keep track of this kind of expressions we fix the following notation.

\begin{nota}
\label{nota.epsilons}
Given non-commutative variables $a, b$, we use the notation $(ab)^1:=ab$ and $(ab)^{*}:=ba$. Given a $n$-tuple $\varepsilon \in \{1,*\}^n$, we denote by $(ab)^\varepsilon:=((ab)^{\varepsilon(1)},(ab)^{\varepsilon(2)},\dots,(ab)^{\varepsilon(n)})$ the $n$-tuple with entries $ab$ or $ba$ dictated by the entries of $\varepsilon$.  Furthermore, we will denote by $(a,b)^{\varepsilon}$ the $2n$-tuple obtained from separating the $a$'s from the $b$'s in the $n$-tuple $(ab)^\varepsilon$.  To keep track of the entries in $(a,b)^{\varepsilon}$ that contain an $a$ we use the notation
\[
A(\varepsilon):=\{ 2i-1 |1\leq i\leq n,\ \varepsilon(i)=1\}\cup \{ 2i |1\leq i\leq n,\ \varepsilon(i)=*\}.
\]
Then the entries in $(a,b)^{\varepsilon}$ that contain a $b$ are given by
$B(\varepsilon):=[2n]\backslash A(\varepsilon)$.

For example, if we consider $\varepsilon=(1,*,*,1,*,1) \in \{1,*\}^6$, then $(ab)^\varepsilon=(ab,ba,ba,ab,ba,ab)$ and if we split each entry we get 
\[
(a,b)^{\varepsilon}=(a,b,b,a,b,a,a,b,b,a,a,b).
\]
This means that
\[
A(\varepsilon)=\{1,4,6,7,10,11\},\qquad\text{and}\qquad B(\varepsilon)=\{2,3,5,8,9,12\}.
\]
\end{nota}

Using the products as entries formula together with the vanishing of mixed cumulants, we can easily rewrite the $n$-th cumulant of the anti-commutator. This is a combination of ideas that is commonly found in arguments involving the combinatorics of free probability.
\begin{prop}
\label{prop.basic.anticomm}
The free cumulants of the anti-commutator $ab+ba$ of two free random variables $a,b$ satisfy the following formula for all $n\in \nn$:
\begin{equation}
\label{eq.rephrase}
\kappa_n(ab+ba) =\sum_{\substack{\pi\in \NN\CC(2n)\\ \pi\lor I_{2n}=1_{2n} }} \sum_{\substack{\varepsilon\in \{1,*\}^{n}\\ \{A(\varepsilon),B(\varepsilon)\}\geq \pi  }} 
\Bigg( \prod_{\substack{V\in \pi, \\ V\subset A(\varepsilon)}} \kappa_{|V|}(a) \Bigg) \Bigg( \prod_{\substack{W\in \pi, \\ W\subset B(\varepsilon)}} \kappa_{|W|}(b) \Bigg),
\end{equation}
where $I_{2n}:=\{\{1,2\},\{3,4\},\dots,\{2n-1,2n\}\}$ is the unique interval pair partition of $[2n]$, and $1_{2n}$ is the maximum partition.
\end{prop}

\begin{proof}
For a fixed $\varepsilon\in \{1,*\}^{n}$, products as arguments formula \eqref{eq.products.arguments} asserts that
\[
\kappa_n((ab)^{\varepsilon})=\sum_{\substack{\pi\in \NN\CC(2n)\\ \pi\lor I_{2n}=1_{2n} }} \kappa_\pi((a,b)^{\varepsilon}).
\]
Therefore, if we sum over all possible $\varepsilon\in \{1,*\}^{n}$ we get that 

\begin{equation}
\label{eq.rephrase.1}
\kappa_n(ab+ba,\dots,ab+ba) =\sum_{\varepsilon\in \{1,*\}^{n}} \sum_{\substack{\pi\in \NN\CC(2n)\\ \pi\lor I_{2n}=1_{2n} }}  \kappa_\pi((a,b)^{\varepsilon})=\sum_{\substack{\pi\in \NN\CC(2n)\\ \pi\lor I_{2n}=1_{2n} }} \sum_{\varepsilon\in \{1,*\}^{n}} \kappa_\pi((a,b)^{\varepsilon}),
\end{equation}
where in the second equality we just changed the order of the sums. Finally, since $a$ and $b$ are free, every mixed cumulant will vanish and thus we require that $\{A(\varepsilon),B(\varepsilon)\}\geq \pi$ in order to get $\kappa_\pi((a,b)^{\varepsilon})\neq 0$. In this case we actually get
\[
\kappa_\pi((a,b)^{\varepsilon})=\Bigg( \prod_{\substack{V\in \pi, \\ V\subset A(\varepsilon)}} \kappa_{|V|}(a) \Bigg) \Bigg( \prod_{\substack{W\in \pi, \\ W\subset B(\varepsilon)}} \kappa_{|W|}(b) \Bigg).
\]
After substituting this into \eqref{eq.rephrase.1} we get the desired result.
\end{proof}

A key idea is that the right-hand side of \eqref{eq.rephrase} can be greatly simplified by observing that there are very few $\varepsilon$ satisfying $\{A(\varepsilon),B(\varepsilon)\}\geq\pi$. Here is where the graph $\GG_\pi$ from Definition \ref{defi.Graph.pi} becomes very useful. Actually, the reason to construct the edges of $\GG_\pi$ by joining the blocks containing $2i-1$ and $2i$ is specifically to store information regarding $\pi\lor I_{2n}$. This is made precise in the next lemma.

\begin{lemma}
\label{Lemma.connected}
Fix a partition $\pi\in \NN\CC(2n)$. Then $\pi\lor I_{2n}=1_{2n}$ if and only if $\GG_\pi$ is connected.
\end{lemma}

\begin{proof}
Recall from Remark \ref{rem.lor.tildelor} that $\pi\lor I_{2n}=1_{2n}$ if and only if $\pi\lor_\PP I_{2n}=1_{2n}$, and this is equivalent to the fact that for every two elements $i,j\in [2n]$ we can find elements $i=i_0,i_1,i_2,\dots, i_{2k+1}=j$ such that for $l=0,\dots,k$, the pair $i_{2l},i_{2l+1}$ is in the same block of $I_{2n}$ and the pair $i_{2l+1},i_{2l+2}$ is in the same block $V_l$ of $\pi$. Observe that $i_{2l},i_{2l+1}$ being in the same block of $I_{2n}$, means that $(i_{2l},i_{2l+1})$ is an edge in $\GG_\pi$ and since $i_{2l+1},i_{2l+2}\in V_l\in \pi$, then edges $(i_{2l},i_{2l+1})$ and $(i_{2l+2},i_{2l+3})$ have the vertex $V_l$ in common. Thus, to every sequence we can associate a path in $\GG_\pi$ between the block containing $i$ and the block containing $j$.
    
Therefore if we assume that $\pi\lor_\PP I_{2n}=1_{2n}$ and want to prove that $\GG_\pi$ is connected, we take two arbitrary blocks $V,W\in \pi$ and some representatives $i\in V$ and $j\in W$. Since $\pi\lor_\PP I_{2n}=1_{2n}$ we have a sequence $i=i_0,i_1,i_2,\dots, i_{2k+1}=j$ and thus a path in $\GG_\pi$ between the $V$ and $W$, so $\GG_\pi$ is connected.
     
Conversely, assume that $\GG_\pi$ is connected, and fix $i,j\in [2n]$. Then we consider the blocks $V,W\in\pi$ such that $i\in V$ and $j\in W$. Since $\GG_\pi$ is connected there exists a path connecting $V$ and $W$ in $\GG_\pi$, which in turn gives us a sequence connecting $i$ and $j$, and thus $\pi\lor_\PP I_{2n}=1_{2n}$.
\end{proof}

Now we can use $\GG_\pi$ to test when $\{A(\varepsilon),B(\varepsilon)\}\geq \pi$ for a fixed $\pi\in \NN\CC(n)$ and a fixed $\varepsilon\in \{1,*\}^{n}$. Turns out that $\GG_\pi$ must be bipartite.

\begin{rem}[Bipartite graphs]
\label{rem.bipartite.graph}
Let $\GG=(V_\GG,E_\GG)$ be an undirected graph (where loops and multiedges are allowed), we say that $\GG$ is \emph{bipartite} if there exist a partition $\{V'_\GG,V''_\GG\}$ of the set of vertices $V_\GG$ such that there are no edges $E_\GG$ connecting two vertices in $V'_\GG$ or connecting two vertices in $V''_\GG$. In other words, all vertices connect a vertex in $V'_\GG$ with a vertex in $V''_\GG$. 

Moreover, if $\GG$ is connected and bipartite, then the bipartition $\{V'_\GG,V''_\GG\}$ is unique (up to the permutation $\{V''_\GG,V'_\GG\}$). This is because once we identify a vertex $v\in V'_\GG$, then the set of any other vertex $u\in  V_\GG$ is determined by its distance to $v$. If the distance is even, then $u\in V'_\GG$, and if the distance is odd, then $u\in V''_\GG$.
\end{rem}

\begin{prop}
\label{prop.few.nonvanishing.epsilons}
Let $\pi\in\NN\CC(2n)$ such that $\pi\lor I_{2n}=1_{2n}$ and consider $\GG_\pi$.
\begin{itemize}
    \item If there exists an $\varepsilon\in\{1,*\}^n$ such that $\{A(\varepsilon),B(\varepsilon)\}\geq \pi$, then $\GG_\pi$ is bipartite.
    \item Moreover, if $\GG_\pi$ is bipartite, then there are exactly two tuples $\varepsilon\in\{1,*\}^n$ satisfying $\{A(\varepsilon),B(\varepsilon)\}\geq \pi$. Furthermore, these tuples are completely opposite, that is, they do not coincide in any entry.
\end{itemize}
\end{prop}

\begin{proof}
For the first part, the existence of an $\varepsilon\in\{1,*\}^n$ such that $\{A(\varepsilon),B(\varepsilon)\}\geq \pi$ implies that we can write $\pi=\{A_1,\dots,A_r,B_1,\dots, B_s\}$ such that $A_i\subset A(\varepsilon)$ for $i=1,\dots, r$ and $B_j\subset B(\varepsilon)$ for $j=1,\dots,s$. If we consider $V'=\{A_1,\dots,A_r\}$ and $V''=\{B_1,\dots, B_s\}$, then $\{V',V''\}$ is a bipartition of $\GG_\pi$. Indeed, by construction of $\GG_\pi$, if $e\in E_{\GG_\pi}$, then $e$ connects the blocks containing elements $2k-1$ and $2k$ for some $k=1,\dots,n$. However, by definition of $A(\varepsilon)$ and $B(\varepsilon)$, one must contain $2k-1$ while the other contains $2k$. Thus $e$ connects vertices from different sets and $\{V',V''\}$ is actually a bipartition of $\GG_\pi$.

For the second part, since $\pi\lor I_{2n}=1_{2n}$ we know from Lemma \ref{Lemma.connected} that $\GG_\pi$ is connected. Then, if $\GG_\pi$ is also bipartite, by Remark \ref{rem.bipartite.graph} there exist a unique bipartition $\{\pi',\pi''\}$ of the vertices of $\GG_\pi$ (blocks of $\pi$), say $\pi'=\{V_1,\dots,V_r\}$ and $\pi''=\{W_1,\dots, W_s\}$, where $V_1$ contains the element 1. Then we must have $A(\varepsilon)=V_1\cup \dots \cup V_r$ or $A(\varepsilon)=W_1\cup \dots \cup W_s$. This clearly determines $\varepsilon$, since $\varepsilon(i)=1$ if and only if $2i-1\in A(\varepsilon)$. Furthermore, since $(V_1\cup \dots \cup V_r)\cap (W_1\cup \dots \cup W_s)=\emptyset$, then the two possible $\varepsilon$ do not coincide in any entry.
\end{proof}

\begin{nota}
In light of the previous result, given a $\pi\in \NN\CC(2n)$ such that $\GG_\pi$ is connected and bipartite, we will denote by $\varepsilon_\pi$ the (unique) tuple such that $\{A(\varepsilon_\pi),B(\varepsilon_\pi)\}\geq \pi$ and $1\in A(\varepsilon)$. And we denote by $\varepsilon'_\pi$ the other possible tuple, which actually satisfies that $A(\varepsilon'_\pi)=B(\varepsilon_\pi)$ and $B(\varepsilon'_\pi)=A(\varepsilon_\pi)$.
\end{nota}

\begin{rem}
\label{rem.pi.prime}
Recall from Definition \ref{main.defi} that every $\pi\in \NCac_{2n}$ is naturally decomposed as $\pi:=\pi'\sqcup \pi''$ where $(\pi',\pi'')$ is the bipartition of $\GG_\pi$. From the previous proof we can observe that we have the equalities $\pi'=\pi|A(\varepsilon_\pi)$ and $\pi''=\pi|B(\varepsilon_\pi)$.  
\end{rem}

We are now ready to prove our main result.

\begin{proof}[Proof of Theorem \ref{Thm.anticommutator.main}]
Our starting point is Equation \eqref{eq.rephrase} from Proposition \ref{prop.basic.anticomm}. Using Lemma \ref{Lemma.connected}, it can be rephrased as
\begin{equation}
\label{eq.2}
\kappa_n(ab+ba) =\sum_{\substack{\pi\in \NN\CC(2n)\\ \GG_\pi\mbox{ is connected} }} \sum_{\substack{\varepsilon\in \{1,*\}^{n}\\ \{A(\varepsilon),B(\varepsilon)\}\geq \pi  }} 
\Bigg( \prod_{\substack{V\in \pi, \\ V\subset A(\varepsilon)}} \kappa_{|V|}(a) \Bigg) \Bigg( \prod_{\substack{W\in \pi, \\ W\subset B(\varepsilon)}} \kappa_{|W|}(b) \Bigg).
\end{equation}
By Proposition \ref{prop.few.nonvanishing.epsilons} the condition $\{A(\varepsilon),B(\varepsilon)\}\geq \pi$ is only true if $\GG_\pi$ is bipartite. Furthermore, in this case $\varepsilon$ can only be one of $\varepsilon_\pi$ or $\varepsilon'_\pi$. Therefore, the right-hand side of \eqref{eq.2} can be simplified to  
\[
\sum_{\substack{\pi\in \NN\CC(2n),\\ \GG_\pi\mbox{ is connected,} \\ \GG_\pi\mbox{ is bipartite} }} 
\Bigg( \prod_{\substack{V\in \pi, \\ V\subset A(\varepsilon_\pi)}} \kappa_{|V|}(a) \prod_{\substack{W\in \pi, \\ W\subset B(\varepsilon_\pi)}} \kappa_{|W|}(b) +
 \prod_{\substack{V\in \pi, \\ V\subset A(\varepsilon'_\pi)}} \kappa_{|V|}(a)  \prod_{\substack{W\in \pi, \\ W\subset B(\varepsilon'_\pi)}} \kappa_{|W|}(b) \Bigg).
\]
Finally, from Remark \ref{rem.pi.prime} we have that $\pi'=\{V\in\pi:V\subset A(\varepsilon_\pi)\}=\{W\in \pi: W\subset B(\varepsilon'_\pi)\}$ and similarly $\pi''=\{W\in\pi:W\subset B(\varepsilon_\pi)\}=\{V\in \pi: V\subset A(\varepsilon'_\pi)\}$, so we obtain the desired formula \eqref{formula.anticommutator.main}.
\end{proof}

\section{Detailed study of the set \texorpdfstring{$\mathcal{Y}_{2n}$}{} }
\label{sec:NCacd}

We begin this section by proving Proposition \ref{prop.NCacd.nice.description}. Then we give some basic remarks and collect various facts concerning the set $\NCacd_{2n}$. The goal is to have a better understanding of this set, which leads to a better understanding of $\NCac_{2n}$ and simplifies the computations when applying Theorem \ref{Thm.anticommutator.main}. The main result in this section is Theorem \ref{thm.recurrence.relation}, that provides a recursive formula to compute the values $|\NCacd_{2n}|$.

In order to apply Theorem \ref{Thm.anticommutator.main} it is useful to have an easy way to determine for which $\pi$ the graph $\GG_\pi$ is connected and bipartite. So far, the only way to check this is to actually construct the graph and analyze it. Here is where Proposition \ref{prop.NCacd.nice.description}, which asserts that $Kr(\NCacd_{2n})=\NCac_{2n}$, proves to be very useful, as the set $\NCacd_{2n}$ has a very simple description.

\begin{proof}[Proof of Proposition \ref{prop.NCacd.nice.description}]
Let us fix $\sigma\in\NN\CC(2n)$ and $\pi:=Kr(\sigma)$. the proof can be separated in two independent statements:

{\bf (A)} $\GG_\pi$ is connected if and only if $\sigma$ separates odd elements (there is no block $W\in \sigma$ containing two odd elements). 

Indeed, from Lemma \ref{Lemma.connected}, $\GG_\pi$ is connected if and only $\pi\lor I_{2n}=1_{2n}$. By properties of the Kreweras complement $\pi\lor I_{2n}=1_{2n}$ if and only if $Kr^{-1}(\pi)\land Kr^{-1}(I_{2n})=Kr^{-1}(1_{2n})=0_{2n}$. Since $Kr^{-1}(I_{2n})=\{\{1,3,\dots,2n-1\},\{2\},\{4\},\dots,\{2n\}\}$, for the latter to hold it is necessary and sufficient that all the odd elements of $\sigma=Kr^{-1}(\pi)$ are in different blocks. 

{\bf (B)} $\GG_\pi$ is bipartite if and only if $|W|$ is even for every block $W\in \sigma$ such that $W\subset\{2,4,\dots,2n\}$. 

In this case it is simpler to show that the negations of previous statements are equivalent. By standard arguments in graph theory (see for instance \cite{bollobas2013modern}) we know that $\GG_\pi$ is non-bipartite if and only if $\GG_\pi$ contains an odd cycle which in turn is equivalent to $\GG_\pi$ containing a simple odd cycle. Thus, to complete the proof we need to show that $\GG_\pi$ contains a simple odd cycle if and only if $\sigma$ contains a block $W$ with $|W|$ odd and $W\subset\{2,4,\dots, 2n\}$.

For the if statement take a simple cycle of $\GG_\pi$ with odd length $j$. Assume that the cycle has (in that order) the vertices $V_1, \dots, V_j\in \pi$  and edges $e_1,e_2,\dots, e_j$ corresponding to $(2i_1-1,2i_1),(2i_2-1,2i_2), \dots, (2i_j-1,2i_j)$. This means that for $k=1,\dots j$, we have that $2i_{k-1},2i_k-1\in V_k$ (with $i_0:=i_j$). Then, the previous restrictions are equivalent to the fact that $W:=\{2i_1,2i_2, \dots, 2i_j\}$ is a block of $\sigma$. Indeed, those elements are in the same block of $\sigma$ since $V_1, \dots, V_j$ are different blocks of $\pi$ (because the cycle is simple). Also $W$ do not contains any other element since this would generate a crossing with one of the arches $(2i_j,2i_1-1),(2i_1,2i_2-1), \dots, (2i_{j-1},2i_j-1)$ of $\pi$. 

Conversely, if $\sigma$ contains the block $W=\{2i_1,2i_2, \dots, 2i_j\}$ of odd size and completely contained in $\{2,4,\dots, 2n\}$. Then each of the arches $(2i_j,2i_1-1),(2i_1,2i_2-1), \dots, (2i_{j-1},2i_j-1)$ belongs to a different block of $\pi$ and this implies that $\GG_\pi$ contains an odd cycle with edges $(2i_1-1,2i_1),(2i_2-1,2i_2), \dots, (2i_j-1,2i_j)$.

Once we have {\bf (A)} and {\bf (B)}, we know that $\pi\in \NCac_{2n}$ if and only if $\sigma$ separates odd elements and $|W|$ is even for every block $W\in \sigma$ such that $W\subset\{2,4,\dots,2n\}$. Then $\sigma$ has blocks $B_1,B_3,\dots,B_{2n-1}$ with $i\in B_i$ for $i=1,3,\dots,2n-1$, and since the rest of the blocks $E_1,\dots,E_r$, are contained in $[2n]\backslash (B_1\cup B_3\cup\dots\cup B_{2n-1})\subset \{2,4,\dots,2n\}$ they must have even size. Therefore we conclude that $\pi \in \NCac_{2n}$ if and only if $\sigma\in \NCacd_{2n}$. 
\end{proof}

Now we turn to study the size of the set $\NCacd_{2n}$. A direct enumeration of this set looks rather complicated, but we can give a recurrence relation. For this we need to introduce the sets $\NCacd_{2n-1}$ for $n\in\nn$, that (analogously to the even case) consist of all non-crossing partitions of $[2n-1]$ that can be written as 
\[
\sigma=\{B_1,B_3,\dots, B_{2n-1},E_1,\dots,E_r\}, \quad \text{with }r\geq 0,
\]
where $i\in B_{i}$ for $i=1,3,\dots,2n-1$ and $|E_j|$ even for $j\leq r$.

\begin{thm}
\label{thm.recurrence.relation}
For every $n\in\nn$, denote $\alpha_n:=|\NCacd_{2n}|$ and $\beta_n:=|\NCacd_{2n-1}|$. Then, we have that $\alpha_1=\beta_1=1$ and the following relations hold:
\begin{equation}
\label{eq.recursion.an}
\alpha_n=\sum_{\substack{ 1\leq s\leq n \\ s \text{ is even} }}  \sum_{\substack{r_1+\dots+r_s=n \\ r_1,\dots,r_s\geq 1 }}  \beta_{r_1}\beta_{r_2}\cdots \beta_{r_s} +\sum_{j=1}^n \beta_{j}\beta_{n+1-j}\qquad \forall n\geq 2,
\end{equation}
\begin{equation}
\label{eq.recursion.bn}
\beta_n=\sum_{s=1}^{n}\sum_{\substack{r_1+\dots+r_s=n, \\ r_1,\dots,r_s\geq 1 }}\beta_{r_1}\beta_{r_2}\cdots \beta_{r_{s-1}}\alpha_{r_s}, \qquad \forall n\geq 2.
\end{equation}
\end{thm}

\begin{rem}
Notice that if we take $\gamma_k=|\NCacd_k|$, then the previous result gives a recursive way to compute $(\gamma_k)_{k\geq 1}$. Since the formulas are different for even and odd values, the notation with $\alpha_n$ and $\beta_n$ turns out to be simpler, and that is the reason we opt for it.
\end{rem}

\begin{proof}
The fact that $\alpha_1=\beta_1=1$ is straightforward since sets $\NCacd_1$ and $\NCacd_2$ just contain one partition, namely $1_1=\{1\}$ and $1_2=\{\{1\},\{2\}\}$, respectively. In order to obtain the recursive formulas the general idea is to represent $\NCacd_m$ as a disjoint union of products of sets $\NCacd_k$ with $k<m$. We will base this study in the following standard bijection between non-crossing partitions:

Given $\sigma\in \NN\CC(m)$, assume that $W=\{w_1,w_2,\dots,w_s\}\in \sigma$ is the block containing the last element, namely $w_s=m$. Consider the intervals $J_1=\{1,2\dots,w_1-1\}, J_2=\{w_1+1,w_1+2,\dots, w_2-1\},\dots, J_s=\{w_{s-1}+1,\dots,m-1\}$ created after removing $W$ from $[m]$, and assume that $|J_i|=k_i$ for $i=1,\dots,s$. Note that if $w_i+1=w_{i+1}$, then $J_i=\emptyset$ and $k_j=0$. Then, we have the following bijection
\begin{align*}
\Psi: \NN\CC(m)& \to \bigcup_{s=1}^m\ \bigcup_{\substack{k_1+\dots+k_s=m-s \\ k_1,\dots,k_s\geq 0 }}  \NN\CC(k_1)\times \cdots \times \NN\CC(k_s),
\\
\sigma &\mapsto \Psi(\sigma):=(\sigma|J_1,\sigma|J_2,\dots,\sigma|J_s),
\end{align*}
where if $k_j=0$, we assume that $\NCacd_{0}=\{\emptyset\}$ contains an abstract \emph{empty partition} and thus $|\NCacd_{0}|=1$. To reconstruct $\sigma$ from an element $(\sigma_1,\sigma_2,\dots,\sigma_s)\in  \NN\CC(k_1)\times \cdots \times \NN\CC(k_s)$ with $k_1+\dots+k_s=m-s$, we consider $w_i=k_1+\dots+k_i+i$ for $i=1,\dots,s$ and define $\sigma$ as the (unique) partition in $\NN\CC(m)$ such that $\{w_1,w_2,\dots,w_s\}\in \sigma$ and $\sigma|J_i=\sigma_i$ for $i=1,\dots,s$.

Here we are only concerned with the restriction of $\Psi$ to the set $\NCacd_m$. Let us separate in the cases $m=2n-1$ and $m=2n$.

{\bf Case $m=2n-1$}. Let $\sigma=\{B_1,B_3,\dots, B_{m},E_1,\dots,E_r\}\in \NCacd_m$ and take the block $B_m=\{w_1,w_2,\dots,w_{s-1},m\}$ containing the last element. Since $\sigma$ separates odd elements, then $w_1,w_2,\dots,w_{s-1}$ must be even. If we consider the intervals $J_1,\dots, J_s$ left by $W$ with sizes $k_1,\dots,k_s$ respectively, we have that $k_1,k_2,k_3,\dots,k_{s-1}$ are odd, while $k_s$ is even. Moreover, since $w_i$ is even for $i=1,\dots,s-1$, the restrictions  $\sigma|J_i$ separate odd elements and the blocks with only even elements are of even size. Therefore, we have that $\sigma|J_i\in \NCacd_{k_i}$ for $i=1,\dots,s$ and we obtain that
\begin{equation}
\label{eq.recurrel.1}
\Psi(\NCacd_m)=\bigcup_{s=1}^m\ \bigcup_{\substack{k_1+\dots+k_s=m-s\\ k_1,k_2,\dots,k_{s-1}\geq 1 \text{ odd} \\ k_s\geq 0 \text{ even} }}  \NCacd_{k_1}\times \cdots \times \NCacd_{k_s}.
\end{equation}
Now, let us write $k_s=2r_s-2$, such that $r_s\geq 1$ and $|\NCacd_{k_s}|=\alpha_{r_s-1}$, we also write $k_j=2r_j-1$ such that $|\NCacd_{k_j}|=\beta_{r_j}$ for $j=1,2,\dots,s-1$. This means that the condition in the indices is $2(r_1+\dots+r_s)-s-1=k_1+\dots+k_s=m-s=2n-1-s$, or equivalently $r_1+\dots+r_s=n$. Finally, since $\Psi$ is bijection and on the right-hand side of \eqref{eq.recurrel.1} the union is disjoint, we obtain \eqref{eq.recursion.bn}.

{\bf Case $m=2n$}. Let $\sigma=\{B_1,B_3,\dots, B_{m-1},E_1,\dots,E_r\}\in \NCacd_m$, then we have two options for the last block $W:=\{w_1,w_2,\dots,w_{s-1},m\}$. Either $W=E_j$ for some $j=1,\dots,r$ (without loss of generality we take $j=1$), or $W=B_i$ for some $i=1,3,\dots,m-1$. 

{\bf First subcase}. Assume that $W=E_1$ and let $\NCacd^{E}_{m}\subset \NCacd_{m}$ be the set of partitions $\sigma\in \NCacd_{m}$ such that their last block, $W$, has only even elements. In this case, $w_1,\dots,w_{s-1},s$ are all even, and then $k_i$ is odd for all $i=1,\dots,s$. Proceeding as in the previous case we get that  
\begin{equation}
\label{eq.recurrel.2}
\Psi(\NCacd^{E}_{m})=\bigcup_{\substack{ 1\leq s\leq m \\ s \text{ is even} }}  \ \bigcup_{\substack{k_1+\dots+k_s=m-s\\ k_1, k_2, \dots, k_s\geq 1 \text{ odd} }} \NCacd_{k_1}\times \cdots \times \NCacd_{k_s}.
\end{equation}
Let us now write $k_i=2r_i-1$ such that $|\NCacd_{k_i}|=\beta_{r_i}$ for $i=1,\dots,s$. The condition on the indices is $2(r_1+\dots+r_s)-s=k_1+\dots+k_s=m-s=2n-s$, which can be restated as $r_1+\dots+r_s=n$, and actually this implies that $s\leq n$. Therefore, we obtain 
\begin{equation}
\label{eq.recursion.an.case1}
|\Psi(\NCacd^{E}_{m})|=\sum_{\substack{ 1\leq s\leq n \\ s \text{ is even} }}  \sum_{\substack{r_1+\dots+r_s=n \\ r_1,\dots,r_s\geq 1 }}  \beta_{r_1}\beta_{r_2}\cdots \beta_{r_s}.
\end{equation}

{\bf Second subcase}. Assume that $W=B_i$ for an odd element $i$, and let $\NCacd^{B_i}_{m}\subset \NCacd_{m}$ be the set of partitions where $W$ is of the form $B_i$. We could proceed as in the previous cases, but now there is a simpler approach. We notice that the restrictions  $\sigma_1:=\sigma|\{1,2,\dots,i\}$ and $\sigma_2:=\sigma|\{i,i+1,\dots m-1\}$ satisfy that $\sigma_1\in \NCacd_{i}$ and $\sigma_2\in\NCacd_{m-i}$. Moreover we have a bijection 
\begin{align*}
\Phi: \NCacd^{B_i}_{m}& \to \NCacd_{i}\times \NCacd_{m-i},
\\
\sigma &\mapsto (\sigma_1,\sigma_2).
\end{align*}
Notice that given $(\sigma_1,\sigma_2)\in \NCacd_{i}\times \NCacd_{m-i}$ such that $\sigma_1=\{V_1,\dots,V_r\}$ with $i\in V_r$, and $\sigma_2=\{W_1,\dots,W_s\}$ with $1\in W_1$. Then the inverse function is given by
\[
\Phi^{-1}(\sigma_1,\sigma_2)=\{ V_r\cup W'_1\cup\{m\}, V_1, V_2,\dots, V_{r-1},W'_2,W'_3,\dots, W'_s\},
\]
where for $k=1,\dots,s$ the sets $W'_k:=W_k+i-1$ are the translations by $i-1$ of the blocks of $\sigma_2$. Finally if we sum over all possible odd numbers $i=2j-1$ between 1 and $m=2n-1$ we obtain that
\begin{equation}
\label{eq.recursion.an.case2}
\sum_{\substack{1\leq i\leq m \\ i \text{ odd} }} |\NCacd^{E_i}_{m}|=\sum_{j=1}^n \beta_{j}\beta_{n+1-j},
\end{equation}

If we add together \eqref{eq.recursion.an.case1} and \eqref{eq.recursion.an.case2} we obtain the desired formula \eqref{eq.recursion.an}.
\end{proof}

Now we want to give a more careful look to the set $\NCacd_{2n}$ and study the cardinality of the subsets where partitions $\sigma\in \NCacd_{2n}$ have a fixed size $|\sigma|=k$. We begin by checking which are these possible sizes $k$.

\begin{rem}
Recall that every $\sigma\in \NCacd_{2n}$ can be written as $\sigma=\{B_1,\dots,B_{2n-1},E_1,\dots,E_r\}$.
Then the size of the partition is $|\sigma|=n+r$ with $r\geq 0$. Moreover, since $|E_j|$ is even we have $|E_j|\geq 2$ for $j=1,\dots, r$, and since $E_1\cup\dots\cup E_r \subset\{2,4,\dots,2n\}$, we have that $|E_1|+\dots+|E_r|\leq n$. Thus, we conclude that $r\leq \tfrac{n}{2}$ and we get a bound for the size of partitions in $\NCacd_{2n}$: 
\[
n\leq |\sigma|\leq \frac{3n}{2}.
\] 
\end{rem}

\begin{nota}
For $\sigma\in \NCacd_{2n}$, we will say that $r:=|\sigma|-n$ is the \emph{level} of $\sigma$ and use the notation
$\NCacd_{2n}^{(r)}:=\{\sigma\in \NCacd_{2n}: |\sigma|=n+r \}$, for $r=0,1,\dots,\left\lfloor \tfrac{n}{2} \right\rfloor$.
\end{nota}

It turns out that we can give closed formulas for the sizes $|\NCacd_{2n}^{(r)}|$ in the extreme cases when the level $r$ is 0 or $\left\lfloor \tfrac{n}{2} \right\rfloor$. Interestingly enough, both formulas are related in a different way to the sequence of Catalan numbers $(C_n)_{n\geq 0}=1,1,2,5,14,\dots$.

\begin{prop}
\label{prop.level.0}
With the previous notation, $|\NCacd_{2n}^{(0)}|=2^{n-1}C_n$. 
\end{prop}

\begin{proof}
Notice that $\NCacd_{2n}^{(0)}$ is the set of non-crossing partitions where every block contains exactly one odd element. We similarly define $\NCacd_{2n-1}^{(0)}$ as the set of non-crossing partitions of $[2n-1]$ where every block contains exactly one odd element. And we use the notation $\theta_m:=|\NCacd_{m}^{(0)}|$. We follow the recurrence relation ideas from the proof of Theorem \ref{thm.recurrence.relation} to inductively show that $\theta_{2n}=2^{n-1}C_n$, and $\theta_{2n+1}=2^{n}C_n$. The induction hypothesis is trivially true, since $|\NCacd_{1}^{(0)}|=|\NCacd_{2}^{(0)}|=1$. For the inductive step we separate in two cases.

First consider the case $m=2n$. Let $\sigma=\NCacd_{2n+1}^{(0)}$ and take the last block $W\in\sigma$ (namely $m\in W$). Opposed to the proof of Theorem \ref{thm.recurrence.relation}, now we cannot have the first subcase $W=E_1$, and only need to check the second subcase $W=B_i$. Moreover, when doing the restrictions we still have exactly one odd element in each block. So we obtain only partitions with level 0. Thus, in this case the bijection $\Phi$ gives

\[
\Phi(\NCacd_{2n}^{(0)})= \bigcup_{\substack{1\leq k\leq 2n\\ k \text{ odd} }} \NCacd_{k}^{(0)} \times \NCacd_{2n-k}^{(0)}.
\]
Thus if we take $k=2j+1$, then the induction hypothesis an the recursive relation satisfied by the Catalan numbers yield that
\[
\theta_{2n}=\sum_{j=0}^{n-1} \theta_{2j+1}\theta_{2n-2j-1}=\sum_{j=0}^{n-1} 2^{j}C_{j} \cdot 2^{n-j-1}C_{n-j-1}=2^{n-1}\sum_{j=0}^{n-1} C_{j-1}C_{n-j}=2^{n-1}C_n.
\]

For the case $m=2n+1$. Let $\sigma=\NCacd_{2n+1}^{(0)}$ and consider the block $W\in\sigma$ with $m\in W$. We separate in 2 cases. If $W=\{m\}$, then we have that $W|\{1,\dots,m-1\}\in \NCacd_{m-1}^{(0)}$ and we obtain $\theta_{2n}=2^{n-1}C_n$ partitions in this subcase. Otherwise, we consider $s=min(W)$, which has to be even (as $s\neq m$ and $\sigma$ separates odd elements). Then we separate $\sigma$ into the partitions $\sigma_1=\sigma|\{1,\dots,s-1\}$ and $\sigma_2=\sigma|\{s+1,\dots,m\}$. Then, the map $\Phi'$ sending $\sigma$ to
\[
\Phi'(\sigma)=(\sigma_1,\sigma_2)\in \bigcup_{\substack{ 1\leq s\leq m \\ s \text{ even} }} \NCacd^{(0)}_{s-1}\times \NCacd^{(0)}_{m-s},
\]
is a bijection. Thus, if we take $s=2j+2$, we count
\[
\sum_{j=0}^{n-1} \theta_{2j+1}\theta_{2n-2j-1}=2^{n-1}C_n
\]
partitions in this subcase. Adding both subcases, we conclude that $\theta_{2n+1}=2^{n}C_n$.
\end{proof}

\begin{prop}
\label{prop.level.max}
With the previous notation, for $k\in \nn$ one has
\[
|\NCacd_{4k}^{(k)}|=C_k\qquad \text{and} \qquad |\NCacd_{4k+2}^{(k)}|=(2k+1)C_{k+1}.
\] 
\end{prop}

\begin{proof}
For the case $2n=4k$, we want to count partitions that can be written as $\sigma=\{B_1,\dots,B_{4k-1},E_1,\dots, E_{k}\}$, but this implies that $|E_1|=|E_2|=\dots=|E_k|=2$ and $|B_1|=\dots=|B_{4k-1}|=1$, and we have a bijection with pair partitions $\{E_1,\dots, E_{k}\}\in \NN\CC_2(2k)$, so we conclude that $|\NCacd_{4k}^{(k)}|=|\NN\CC_2(2k)|=C_k$. 

For the case $2n=4k+2$ we use the same kind of bijections we have been using throughout all this section. We want to count partitions of the form $\sigma=\{B_1,\dots,B_{4k+1},E_1,\dots, E_{k}\}$, but this implies that $|E_1|=|E_2|=\dots=|E_k|=2$ and $|B_1|=\dots=|B_{4k-1}|=1$, except for one block $B_i=\{e,i\}$ that consist of its corresponding odd element $i$ and one even element $e$. Then, the block $B_i$ separates the partition $\sigma$ into two independent partition $\sigma_1$ and $\sigma_2$ contained in the pockets $J_1=\{e+1,\dots, i-1\}$ and $J_2=\{i+1,\dots, e-1\}$ (with elements taken mod $4k$). Moreover, the even elements of $\sigma_1$ form a pair partition, so $|\sigma_1|:=4j$ should be multiple of 4, and we actually have that $\sigma_1\in \NCacd_{4j}^{(j)}$ and $\sigma_2\in \NCacd_{4k-4j}^{(k-j)}$. Thus, there are $2k+1$ ways to choose the odd element $i$, and for each of them we have 
\[
\sum_{j=0}^k |\NCacd_{4j}^{(j)}|\cdot |\NCacd_{4k-4j}^{(k-j)}|=\sum_{j=0}^k C_jC_{k-j}=C_{k+1}
\]
ways to construct the partitions $\sigma_1$ and $\sigma_2$ and we conclude that $|\NCacd_{4k+2}^{(k)}|=(2k+1)C_{k+1}$.
\end{proof}

\begin{rem}
At this point, one may wonder if the sizes of the sets $|\NCacd_{2n}^{(r)}|$ for $r=1,\dots,\left\lfloor \tfrac{n}{2} \right\rfloor-1$, have some other interesting formulas. We can surely generalize the ideas used in Theorem \ref{thm.recurrence.relation}, Proposition \ref{prop.level.0} and Proposition \ref{prop.level.max} to get recursive formulas. However, the sets $\NCacd_{2n}^{(r)}$ (and thus the recursions needed) become much more involved and it is not clear if there is a nice closed formula. We now point out an algorithm that can be used to recursively compute each value. Observe that $|\NCacd_{2n}^{(1)}|$ can be written in terms of $|\NCacd_{m}^{(1)}|$ and $|\NCacd_{m}^{(0)}|$ for $m < 2n$. Since we know the values for the level $r=0$ we can recursively compute all the values for level $r=1$. In general, $|\NCacd_{2n}^{(r)}|$ can be written in terms of $|\NCacd_{m}^{(j)}|$ for $m < 2n$ and $j\leq r$. Thus we can do a recursion on $r$ to compute all the values of $|\NCacd_{m}^{(r)}|$ (where for each $r$ we first do the recursion on $m$). 
\end{rem}

The sizes of the sets $|\NCacd_{2n}|$ and $|\NCacd_{2n}^{(r)}|$ for small values of $n$ and $r$ are given in the following table:

\[
\begin{array}{c|ccccc}
2n &  2 &  4 &  6 &  8 &  10 \\ \hline 
|\NCacd_{2n}|   & 1  & 5 &  26  & 155    & 987 \\ \hline
|\NCacd^{(0)}_{2n}| &     1    &    4   &     20   &    112    &  672  \\
|\NCacd^{(1)}_{2n}| &      &    1 &  6  &   41  &  290  \\
|\NCacd^{(2)}_{2n}| &       &     &  &    2  & 25  
\end{array}
\]

As the reader may have noticed, it is useful to consider the sizes $|\NCacd_{m}|$ for all integers $m$, rather than just the even ones. Thus, we also present a table including these values (the level in the odd case is defined in the same way as in the even case).

\[
\begin{array}{c|ccccccccccc}
m & 1 & 2 & 3 & 4 & 5 & 6 & 7 & 8 & 9 & 10 & 11  \\ \hline 
|\NCacd_m|   &  1    &    1&  2    & 5 &     9   & 26  &   48 & 155    &   287    & 987 & 1834   \\  \hline 
|\NCacd^{(0)}_m|   &   1   &      1    &    2    &   4   &   8   &   20   &  40  &  112  &  224  &  672 & 1344 \\  
|\NCacd^{(1)}_m|    &         &      &     &    1 & 1   & 6  & 8  &  41  &  61  & 290 & 460 \\
|\NCacd^{(2)}_m|   &    &     &    &    &     &  & &    2 & 2 & 25 & 30 
\end{array}
\]

\begin{rem}
Since $|\sigma|+|Kr(\sigma)|=2n+1$, all the previous discussion directly gives information on the size of the levels of $\NCac_{2n}$. In particular, for every $\pi\in \NCac_{2n}$ we have that
\[
n+1\geq |\pi|\geq \left\lceil \tfrac{n}{2} \right\rceil+1.
\]
And for a fixed $k\in\nn$ we have that $Kr(\NCacd_{2n}^{(n+1-k)})=\{\sigma\in \NCac_{2n}: |\sigma|=k \}$, and thus their cardinalites concide.
\end{rem}
    
To finish this section we give a possible line of study that will help to further understand the set $\NCacd_{2n}$.

\begin{rem}
Fix a $\sigma\in \NCacd_{2n}$ and consider its restriction $\sigma_{even}:= \sigma|\{2,4,\dots,2n\}$ to the even elements. This gives us a non-crossing partition of $[n]$. Moreover, it is easy to observe that for every $\pi\in\NN\CC(n)$ we can find a $\sigma\in \NCacd_{2n}$ such that $\sigma_{even}=\pi$. Then, for a fixed $\pi\in \NN\CC(n)$, an interesting question is which are the values $q_\pi:=|\{\sigma \in \NCacd_{2n}:\sigma_{even}=\pi\}|$. This information will give a much better understanding of $\NCacd_{2n}$ and $\NCac_{2n}$. For instance if $\pi=0_n$, the partitions of $\sigma \in \NCacd_n$ such that $\sigma_{even}=0_n$ are exactly the pair partitions. Thus, in this case we get $q_{0_n}=|\NN\CC(n)|=C_{n}$ partitions. We now give some basic ideas towards finding the values $q_\pi$.

In general, given $\pi\in \NN\CC(n)$, a possible algorithm to construct a $\sigma \in \NCacd_{2n}$ such that $\sigma_1=\pi$ is the following. We take $[2n]$ and draw $\pi$ in the even elements $\{2,4,\dots,2n\}$. Then we choose a subset $S=\{V_1,\dots,V_r\}$ of the blocks of $\pi$, and attach to each block a distinct odd element. For the resulting partition to be in $\NCacd_{2n}$ we need to make sure of two things. First, that all the blocks of odd size in $\pi$ are in $S$. The second and more tricky part, is that we do not generate any crossing when doing the attachment of the odd elements.

An attachment that works for every partition is to put the element $min(V)-1$ in each $V\in S$. Alternatively, we can attach $min(V)+1$ to each $V\in S$. Thus there are at least two options for each subset $S$. Since there are $2^{|\{V\in \pi: |V|\text{ is even}\}|}$ possible ways to pick the subset $S$ we get the following lower bound:
\[
q_\pi \geq 2^{|\{V\in \pi: |V|\text{ is even}\}|+1}.
\]
This bound is tighter when $\pi$ has various even blocks. However, if $\pi$ has only blocks of odd size, such as when $\pi=0_n$, the right-hand side becomes 2 and the bound does not really helps. Knowing the values $q_\pi$, or at least improving this simple bound is a useful step towards a better understanding of $\NCacd_{2n}$.
\end{rem}


\section{Applications}
\label{sec:other.applications}

This section is divided in two parts, each makes use of Theorem \ref{Thm.anticommutator.main} and some results from Section \ref{sec:NCacd} to deal with different applications. In Section \ref{sec:free.poisson.1} we study the case of two free Poisson of parameter 1 and prove Theorem \ref{thm.inverse.moment.series}. Then, in Section \ref{subsec:anti-commutator.even.elements} we use our method to retrieve the formula from \cite{nica1998commutators}, where distributions are assumed to be symmetric.

\addtocontents{toc}{\SkipTocEntry}
\subsection{Anti-commutator of two free Poisson with
parameter 1}
\label{sec:free.poisson.1}

$\ $

\noindent 
The objective of this section is to study the distribution $\nu$ of the anti-commutator $ab+ba$ in the particular case where $a,b$ are distributed as a free Poisson distribution (also known as Marchenko-Pastur distribution) of parameter 1, and in particular prove Theorem \ref{thm.inverse.moment.series}. The reason why the free Poisson is very special, is because all  cumulants are equal to 1, namely $\kappa_n(a)=\kappa_n(b)=1$ for all $n\in\nn$. 

\begin{rem}
Notice that Corollary \ref{cor.anticomm.free.poisson.lambda} directly yields that
\[
\kappa_n(ab+ba)= 2 |\NCac_{2n}|=2 | \NCacd_{2n}|, \qquad \forall n\in\nn.
\]

Therefore, we can use formulas \eqref{eq.recursion.an} and \eqref{eq.recursion.bn} from Theorem \ref{thm.recurrence.relation} to recursively compute the cumulants of the anti-commutator $ab+ba$. The first few cumulants are

\[
2, \ 10, \ 52, \  310, \ 1974, \ 13176, \  90948, \ 643918, \ 4650382, \dots 
\]
Therefore, the moment-cumulant formula yields that the first few moments are
\[
2, \  14, \ 120, \  1182, \  12586, \ 141160, \  1642584, \  19646558, \ 240050838, \dots 
\]
\end{rem}

We now want to compute the moment series of $\nu$. In order to do this we first rewrite Theorem \ref{thm.recurrence.relation} as a relation between the formal power series associated to $(\alpha_n)_{n\geq 1}$ and $(\beta_n)_{n\geq 1}$, the sequences of values of $\NCacd_{m}$ for even and odd $m$ respectively. 

\begin{nota}
Let us denote by $A$ and $B$ the formal power series on $x$ associated to the sequences $(\alpha_n)_{n\geq 1}$ and $(\beta_n)_{n\geq 1}$, respectively. That is
\[
A(x):=\sum_{n=1}^\infty \alpha_n x^n=\sum_{n=1}^\infty |\NCacd_{2n}| x^n, \qquad B(x):=\sum_{n=1}^\infty \beta_n x^n=\sum_{n=1}^\infty |\NCacd_{2n-1}| x^n.
\]
\end{nota}

\begin{prop}
The formal power series $A(x)$ and $B(x)$ satisfy the following relations
\begin{equation}
\label{eq.series.recursion.A}
A(x)=\frac{B^2(x)}{1-B^2(x)}+\frac{B^2(x)}{x},
\end{equation}
\begin{equation}
\label{eq.series.recursion.B}
B(x)=\frac{(A(x)+1)x}{1-B(x)},
\end{equation}
\end{prop}

\begin{proof}
First we observe that basic operations of power series tell us that
\[
\frac{1}{1-B(x)}=1+B(x)+B^2(x)+\dots=1+ \sum_{n=1}^\infty x^n \sum_{s=1}^n \sum_{\substack{r_1+\dots+r_s=n \\ r_1, r_2,\dots,r_s\geq 1 }}  \beta_{r_1}\beta_{r_2}\dots \beta_{r_s}.
\]

To handle the sum appearing in \eqref{eq.recursion.an.case1}, where tuples must be of even size we use the following formal power series:

\[
E(x):=\sum_{n=1}^\infty x^n \sum_{\substack{1\leq s\leq n\\ s \text{ is even}}} \sum_{\substack{r_1+\dots+r_s=n \\ r_1, r_2,\dots,r_s\geq 1 }}   \beta_{r_1} \beta_{r_2}\dots  \beta_{r_s}.
\]
A straightforward comparison of coefficients yields $E(x)=B^2(x)(E(x)+1)$. Solving for $E$ we get
\[
E(x)=\frac{B^2(x)}{1-B^2(x)}.
\]

Then, the $n$-th coefficient of $A(x)$ is the left hand side of  \eqref{eq.recursion.an}, while the right-hand side is exactly the sum of $n$-th coefficient of $E(x)$ plus the $n$-th coefficient of $\frac{B^2(x)}{x}$, thus we get formula \eqref{eq.series.recursion.A}:
\[
A(x)=E(x)+\frac{B^2(x)}{x}=\frac{B^2(x)}{1-B^2(x)}+\frac{B^2(x)}{x}.
\]
On the other hand if we multiply $(\frac{1}{1-B(x)})(A(x)+1)(x)$ (where $x$ merely just adjust the coefficients), then the $n$-th coefficient of this product is given by the right-hand side of \eqref{eq.recursion.bn}, since the left hand side is $n$-th coefficient of $B(x)$ we obtain \eqref{eq.series.recursion.B}.
\end{proof}

Since we have two power series, satisfying two equations, \eqref{eq.series.recursion.A} and \eqref{eq.series.recursion.B}, we can manipulate them to obtain functional equation containing only $A(x)$ or only $B(x)$. After these manipulations we obtain the following.

\begin{prop}
The formal power series $A(x)$ and $B(x)$ satisfy the following equations:
\begin{equation}
\label{eq.only.series.A}
4(A(x)+1)^4x^2+7(A(x)+1)^3x-4(A(x)+1)^2x-2(A(x)+1)^2+(A(x)+1)+1=0,
\end{equation}
\begin{equation}
\label{eq.only.series.B}
B(x)(1-2B(x))(1-B(x))(1+B(x))=x.
\end{equation}
\end{prop}

\begin{proof}
To simplify notation we omit the dependence on $x$ whenever we are just evaluating on $x$, that is, we write $A,B$ instead of $A(x),B(x)$. We also use the notation $A_1:=A(x)+1$. By using \eqref{eq.series.recursion.B} and then \eqref{eq.series.recursion.A} we get that 
\[
B-B^2= A_1x=\pai \frac{B^2}{1-B^2}+\frac{B^2}{x}+1\pad x=\pai \frac{1}{1-B^2}+\frac{B^2}{x}\pad x=\frac{x}{1-B^2}+B^2.
\]
Multiplying by $1-B^2$ and solving for $x$ yields \eqref{eq.only.series.B}. On the other hand, from \eqref{eq.series.recursion.B} and \eqref{eq.only.series.B} we obtain
\[
\frac{1}{A_1}=\frac{x}{B(1-B)}=(1-2B)(1+B)=1-B-2B^2.
\]
Also, directly from \eqref{eq.series.recursion.B} we know that $A_1x=B-B^2$, a linear combination of both equations yields
\[
\frac{1}{A_1}-2A_1x=1-B-2B^2-2B+2B^2=1-3B.
\]
Solving for $B$ we obtain 
\[
B=\frac{1}{3}\pai 1- \frac{1}{A_1}+2A_1x\pad =\frac{2A_1^2x+A_1-1}{3A_1} .
\]

Replacing this on \eqref{eq.series.recursion.B} we obtain an expression entirely on $A_1$:
\begin{eqnarray*}
A_1x&=&\pai\frac{-2A_1^2x+2A_1+1}{3A_1}\pad \pai\frac{2A_1^2x+A_1-1}{3A_1}\pad \\ &=&\frac{-4A_1^4x^2-2A_1^3x+2A_1^2x+4A_1^3x+2A_1^2-2A_1+2A_1^2x+A_1-1}{9A_1^2}.
\end{eqnarray*}
Multiplying by $9A_1^2$ and simplifying we obtain \eqref{eq.only.series.A}.
\end{proof}

We are now ready to prove Theorem \ref{thm.inverse.moment.series}, the main result of this section. We use the standard notation
\[
R_\nu(z):=\sum_{n=1}^\infty \kappa_n(ab+ba) z^n\qquad\text{and}\qquad M_\nu(z):=\sum_{n=1}^\infty \phi((ab+ba)^n) z^n,
\]
for the $R$-transform and moment series of $\nu$, respectively. Where $\nu$ is the distribution of the free anti-commutator of two free Poisson distributions. 

\begin{proof}[Proof of Theorem \ref{thm.inverse.moment.series}]
We want to show that the compositional inverse of the moment series of $\nu$ is given by
\begin{equation}
  M^{\langle-1\rangle}_{\nu}(z)=\frac{-7z-6 + 3 \sqrt{(z+2)(9z+2) }}{4(z+2)^2(z+1)}.
\end{equation}

For this, we first observe that the $R_\nu$ is simply $2\cdot A$. This implies that $2(A(x)+1)=R_\nu(x)+2$. Therefore, if we multiply \eqref{eq.only.series.A} by 8 (to avoid fractions on the coefficients) and write it in terms of the $R$-transform we get
\begin{align*}
 0&=2(R_\nu(x)+2)^4x^2+7(R_\nu(x)+2)^3x-4(R_\nu(x)+2)^2x-8(R_\nu(x)+2)^2+4(R_\nu(x)+2)+8 \\
 &=\Big(2(R_\nu(x)+2)^4\Big)x^2+\Big((R_\nu(x)+2)^2(7R_\nu(x)+6)\Big) x -\Big( 4R^2_\nu(x)+12R_\nu(x)\Big).
\end{align*}

Let us simplify this expression by taking $y:=(R_\nu(x)+2)^2x$, this yields
\[
0=2y^2+(7R_\nu(x)+6)y -4(R^2_\nu(x)+3R_\nu(x)).
\]
This is a polynomial in $y$ of degree 2 so by the quadratic formula we know that 
\[
y=\frac{-(7R_\nu(x)+6) \pm \sqrt{(7R_\nu(x)+6)^2+4\cdot 2\cdot 4(R^2_\nu(x)+3R_\nu(x))}}{2\cdot 2}.
\]
The discriminant simplifies to
\[
49R^2_\nu(x)+84R_\nu(x)+36+32R^2_\nu(x)+96R_\nu(x)
=9(R_\nu(x)+2)(9R_\nu(x)+2),
\]
and we obtain that
\[
x= \frac{y}{(R_\nu(x)+2)^2} =\frac{-(7R_\nu(x)+6) \pm 3\sqrt{(R_\nu(x)+2)(9R_\nu(x)+2) }}{4(R_\nu(x)+2)^2}.
\]
Since we want $R_\nu(0)=0$ to hold, we pick the positive sign in the previous equation. Thus the compositional inverse $R^{\langle-1\rangle}(z)$ of the $R$-transform is
\[ 
R^{\langle-1\rangle}(z)= \frac{-7z-6 + 3\sqrt{(z+2)(9z+2) }}{4(z+2)^2}.
\]
And by a basic relation between $R^{\langle-1\rangle}(z)$ and $M^{\langle-1\rangle}(z)$, (see formula (16.31) of \cite{NS}), we conclude that
\[
M^{\langle-1\rangle}(z)= \frac{R^{\langle-1\rangle}(z)}{(1+z)}=\frac{-7z-6 +3 \sqrt{(z+2)(9x+2) }}{4(z+2)^2(z+1)},
\]
as desired.
\end{proof}

\begin{rem}
If we consider the Cauchy transform $G:=G_\nu(z)$ and use its relation to the $R$-transform or the Moment series of $\mu$, we can obtain that $G$ is the solution to a polynomial equation of degree six:
\[
2z^4G^6+8z^3G^5+12z^2G^4+8zG^3+2G^2+7z^3G^4+13z^2G^3+5zG^2-G-4z^2G^2-4zG+8=0.
\]
For instance, we can directly obtain this from \eqref{eq.only.series.A} and the relation $zG(z)+1=2A(G(z))+2$.
\end{rem}

\begin{rem}[Anti-commutator of two free Poisson with parameter $\lambda$]
\label{rem:free.poisson.lambda}
Let $\nu$ be the distribution of $ab+ba$, where $a,b$ are free Poisson of parameter $\lambda>0$. Notice that our discussion from Section \ref{sec:NCacd} together with Corollary \ref{cor.anticomm.free.poisson.lambda} yield that $\kappa_n (\nu) = P_n(\lambda)$ is a polynomial on $\lambda$ of degree $n+1$, with principal coefficient $2^n C_n$, and where 0 is a root of multiplicity $\left\lceil \tfrac{n}{2} \right\rceil+1$. More specifically, if we let $h:=\left\lfloor \tfrac{n}{2} \right\rfloor$ we have that
\[
P_n(\lambda)=\lambda^{n+1-h} \Big( d_0 \lambda^{h} + d_1 \lambda^{h-1}+\dots + d_{h}\Big),
\]
where $d_r:=2|\NCacd_{2n}^{(r)}|$ for $r=0,1,\dots,h$.
\end{rem}

\addtocontents{toc}{\SkipTocEntry}
\subsection{Anti-commutator of two even elements}
\label{subsec:anti-commutator.even.elements}

$\ $

\noindent 
In this section we apply Theorem \ref{Thm.anticommutator.main} to the special case where both $a$ and $b$ are even elements. Recall that $a$ is an \emph{even element} if all its odd moments are 0, namely $\varphi(a^n)=0$ for every odd $n$. This case was already studied by Nica and Speicher in connection with the free commutator (see Theorem 15.20 of \cite{NS}). Here we recover this formula using our main result and some nice combinatorial observations.

\begin{thm}[Theorem 15.20 of \cite{NS}]
\label{thm.anticommutator.even.variables}
Let $a$ and $b$ two free even random variables. Then the odd free cumulants of  the anti-commutator are 0 while the even free cumulants are given by the following formula
\begin{equation}
\label{eq.anticommutator.even}
\kappa_{2n}(ab+ba) =2\sum_{\pi_1\in \NN\CC(n)}  \pai \prod_{V\in \pi_1} \kappa_{2|V|}(a) \pad \sum_{\substack{\pi_2\in \NN\CC(n) \\ \pi_2 \leq Kr(\pi_1)}}  \pai \prod_{W\in \pi_2} \kappa_{2|W|}(b) \pad, \qquad \forall n\in\nn.
\end{equation}
\end{thm}

\begin{proof}
The approach is to use our main theorem, to express the cumulant of the anti-commutator as sum indexed by $\NCac_{2n}$, that can be further restricted to even partitions. If we apply $Kr^{-1}$ to this index set, we obtain partitions in $\NCacd_{2n}$ that are parity preserving. These partitions have a simple description, so when we apply the Kreweras complement we retrieve the result. Let us begin by applying Theorem \ref{Thm.anticommutator.main} to rewrite $\kappa_n(ab+ba)$ as
\begin{equation}
\label{eq.anticomm.formula.even}
\sum_{\substack{\pi\in \NCac_{2n} \\ \pi=\pi'\sqcup\pi''}} \Bigg( \prod_{V\in \pi'} \kappa_{|V|}(a)  \prod_{W\in \pi''} \kappa_{|W|}(b) + \prod_{V\in \pi'} \kappa_{|V|}(b)  \prod_{W\in \pi''} \kappa_{|W|}(a) \Bigg).  
\end{equation} 

Recall that by the moment-cumulant formula, since $a$ and $b$ are even, this readily implies that all its odd cumulants vanish. Thus, whenever we have some block $V\in\pi$ such that $|V|$ is odd, the whole product will vanish. As a result, if $n$ is odd, all the products will have such a block and we get that $\kappa_n(ab+ba)= 0$. Moreover when $n$ is even, the partition $\pi$ must be even for the product not to vanish. Thus the index of the sums in \eqref{eq.anticomm.formula.even} can be restricted to $\NCac_{2n}\cap \NCeven(2n)$. Since we understand better $\NCacd_{2n}$, let us take inverse of the Kreweras complement map. Using Remark \ref{rem.even.parity.preserving} we get that
\[
Kr^{-1}(\NCac_{2n}\cap \NCeven(2n))=Kr^{-1}(\NCac_{2n})\cap Kr^{-1}(\NCeven(2n))=\NCacd_{2n} \cap \NCpp(2n).
\]

Notice that every partition $\sigma\in \NCacd_{2n} \cap \NCpp$ can be written as $\{B_1,\dots,B_{2n-1},E_1,\dots, E_r\}$ with $i\in B_i$ and $|E_j|$ even. Moreover, the blocks of $\sigma$ have the same parity. Therefore, $B_i\subset \{1,3,\dots,2n-1\}$ and this implies that $B_i=\{i\}$ for $i=1,3,\dots, 2n-1$. On the other hand, we have that $E_1\cup\dots \cup E_r=\{2,4,\dots,2n\}$ and since all have even size, we get that $\sigma':=\sigma|\{2,4,\dots,2n\}$ is an even partition. Therefore, $\sigma=\lan 0_n,\sigma'\ran$ for an even partition $\sigma'$ (see Notation \ref{nota.odd.even.partition}), and we obtain the simpler description
\[
\NCacd_{2n} \cap \NCpp(2n)=\{\lan 0_n,\sigma'\ran \in \NN\CC(2n): \sigma'\in \NCeven(n)\}.
\]

Now we apply the Kreweras complement to this set in order to retrieve $\NCac_{2n}\cap \NCeven(2n)$. Let us denote by $I'_{2n}:=\{\{2,3\},\{4,5\},\dots,\{2n-2,2n-1\},\{1,2n\}\}$ the partition which is the Kreweras complement of $\lan 0_n,1_n\ran=\{\{1\},\{3\},\dots,\{2n-1\},\{2,4,\dots,2n\}\}$. For $\sigma=\lan 0_n,\sigma'\ran$ we observe that if $\sigma \leq \lan 0_n,1_n\ran$ then $Kr(\sigma)\geq I'_{2n}$. Moreover the fact that $\sigma'\in \NCeven(n)$ implies that $Kr(\sigma)|\{2,4,\dots,2n\}\in \NCpp(n)$. 

Let us now consider the `fattening' map $\Psi:[n]\to I'_{2n}$ that sends every $j\in [n]$ to the pair $\Psi(j)=\{2j,2j+1\}\in I'_{2n}$ (where we assume $2n+1=1$). We can naturally extend this map to blocks $V=\{j_1,\dots,j_k\}\subset[n]$ by taking $\Psi(V)=\{2j_1,2j_1+1,\dots,2j_k,2j_k+1\}\subset[2n]$, and we can extend it further to partitions $\tau\in\NN\CC(n)$ by taking $\Psi(\tau)=\{\Psi(V):V\in\tau\}\in \NN\CC(2n)$. Then we have a bijection
\[
\Psi: \NN\CC(n) \to \{\pi\in \NN\CC(2n):\pi\geq I'_{2n}\},
\]
and we obtain that partitions $\pi\in\NCac\cap \NCeven=Kr(\NCacd \cap \NCpp)$ can be described as $\pi=\Psi(\tau)$ for some $\tau\in\NCpp(n)$.
Therefore in \eqref{eq.anticomm.formula.even} we can alternatively sum over $\Psi(\NCpp(n))$. Moreover, the tuple $\varepsilon_\pi$ associated to a partition $\pi\in \Psi(\NCpp(n))$ satisfies that $A(\varepsilon_\pi)=\{\Psi(V)\subset[2n]: V \subset\{1,3,\dots,n-1\}\}$ and $B(\varepsilon_\pi)=\{\Psi(V)\subset[2n]: V \subset\{2,4,\dots,n\}\}$. Thus, for $\pi=\Psi(\tau)$ we have
\[
\prod_{\substack{V\in \pi, \\ V\subset A(\varepsilon_\pi)}} \kappa_{|V|}(a)\prod_{\substack{W\in \pi, \\ W\subset B(\varepsilon_\pi)}} \kappa_{|V|}(b)=\prod_{\substack{V'\in \tau, \\ V'\subset \{1,3,\dots,n-1\} }} \kappa_{2|V'|}(a) \prod_{\substack{W'\in \tau, \\ W'\subset \{2,4,\dots,n\}}} \kappa_{2|W'|}(b).
\]
Finally, if we write $n=2m$ we have the description $\NCpp(n):\{\lan\pi_1,\pi_2\ran\in \NN\CC(n):  \pi_1,\pi_2\in\NN\CC(m),\ \pi_2\leq Kr(\pi_1)\}$, and we obtain the desired result.
\end{proof}

\section{Cacti graphs and quadratic forms}
\label{sec:quadratic.forms}

This section is divided in two parts. The first part is devoted to prove Theorem \ref{thm.anticommutator.graphs}, which is a formula for the anti-commutator where the summation is now indexed by cacti graphs. In the second part we prove Theorem \ref{thm.quadratic.forms}, a generalization that expresses cumulants of quadratic forms as sums indexed by colored cacti graphs.  

\addtocontents{toc}{\SkipTocEntry}
\subsection{A formula in terms of cacti graphs}

$\ $

\noindent 
Through this section we will consider a directed version of the graph $\GG_\pi$ from Definition \ref{defi.Graph.pi}.

\begin{defi}
Given a $\pi\in \NN\CC(n)$, we denote by $\overrightarrow{\GG_\pi}$ the directed graph with vertices given by the blocks of $\pi$ and $n$ edges, where for the direction of the $k$-th edge, we simply select the block containing element $2k-1$ as the outgoing vertex and the block containing element $2k$ as the ingoing vertex.

For a directed graph $\overrightarrow{\GG}$, we use the notation $\overrightarrow{(v,w)}$ to mean that $v$ is the outgoing vertex and $w$ the ingoing vertex. We say that a cycle $v_1,\dots,v_j$ is \emph{oriented} if all its edges are in the same direction, namely $\overrightarrow{(v_i,v_{i+1})}$ for $i=1,\dots,j$.
\end{defi}

We begin by proving Proposition \ref{prop.cactus}, which asserts that graphs $\GG_\pi$ coming from a non-crossing partition $\pi$ must be cacti. Moreover, we notice that the cycles of $\overrightarrow{\GG_\pi}$ are oriented.

\begin{proof}[Proof of Proposition \ref{prop.cactus}]
Let $\pi$ be a non-crossing partition and take a block $V\in \pi$. The other blocks $W\in\pi$ can be naturally separated in two groups. Either $W$ is nested inside of $V$, namely $\min{V}\leq \min{W}\leq \max{W} \leq \max{V}$, otherwise we say the $W$ is outside $V$. Let us denote 
\[
\NN_V:=\{W\in\pi: W \text{ nested inside of }V\} \qquad \text{and}\qquad \OO_V:=\pi\backslash (\NN_V\cup \{V\}).
\]

Notice that there is no edge in $\GG_\pi$ that connects a vertex from $\NN_V$ with a vertex from $\OO_V$. Indeed, for the sake of contradiction assume that $Y\in\NN_V$, $W\in\OO_V$ and the edge $e_i=(2i-1,2i)$ connects $Y$ and $W$, this means that $Y$ and $W$ have elements that differ by 1, but this is impossible as $Y\subset \{\min(V)+1,\dots, \max(V)-1\}$ while $W\subset \{1,\dots \min(v)-1\}\cup\{\max(V)+1,\dots, 2n\}$, a contradiction. 

If we consider a simple cycle $V_1,\dots,V_j$ of $\GG_\pi$, the previous observation yields that for $i=1,\dots,j$ if we take the vertex $V_i$, then the other vertices $V_1,\dots, V_{i-1},V_{i+1}, \dots,V_j$ all belong either to $\NN_{V_i}$ in which case we say $V_i$ is principal, or else they all belong to $\OO_{V_i}$. It is easy to check that there must be at least one principal block, otherwise we can order the blocks $V_1<V_2<\dots< V_j$ in such a way that $\max(V_{i})+1=\min(V_{i+1})$ for $i=1,\dots j-1$. But we also need that $\max(V_{j})+1=\min(V_{1})$ which is impossible. Moreover, if a vertex $V$ is principal, then the other blocks are nested inside it and no other vertex can be principal, so in every cycle there is exactly one principal vertex. Assume without loss of generality that $V_1$ is the principal vertex, this forces that $\max(V_{i})+1=\min(V_{i+1})$ for $i=2,\dots j-1$, and that $\min(V_{2})-1$ and $\max(V_j)+1$ are two consecutive elements of $V_1$.  Notice that if we consider the directed version $\overrightarrow{\GG_\pi}$, this implies that the cycle is oriented, thus every simple cycle of $\overrightarrow{\GG_\pi}$ must be oriented.

Observe that a principal vertex $V$ is a block of the partition $\pi$, thus its elements $V=\{x_1,\dots,x_r\}$ naturally leave $r-1$ pockets $P_s:=\{x_s+1,\dots x_{s+1}-1\}$ for $s=1,\dots, r-1$. Then, if an edge $(W,Y)$ is on a simple cycle with principal vertex $V$, this means that both blocks are contained on the same pocket, say $P_i$. Moreover, the only possible cycle within $P_i$ that has $V_1$ as principal vertex must consist of those blocks covered by $V$ in $P_i$ (namely, the outer blocks of $\pi|P_i$). This implies that given a vertex $V$ and an edge $(W,Y)$ of a graph $\GG_\pi$, there is at most one simple cycle that has $V$ as principal vertex and $(W,Y)$ as an edge.

To conclude that every $\GG_\pi$ is a cactus we proceed by contradiction. Assume that the edge $(Y,W)$ belongs to two different simple cycles $s_1$ and $s_2$, and consider the principal vertex $V_1$ and $V_2$ of each cycle. Then $V_1$ and $V_2$ both cover $Y$ (and $W$) so $V_1=V_2$, but this further implies that $s_1=s_2$.
\end{proof}

\begin{rem}
Recall that for a planar graph $\GG$, Euler's formula states that $v-e+f=2$, where $v:=|V_G|$, $e:=|E_\GG|$ and $f$ are the number of vertices, edges and faces, respectively, of the graph $\GG$. This formula has a very nice interpretation in terms of partitions. Consider a partition $\sigma\in \NCacd_{2n}$, its Kreweras complement $\pi:=Kr(\sigma)\in \NCac_{2n}$ and the graph $\GG_\pi$. First, the vertices of $\GG_\pi$ are the blocks of $\pi$, so $v=|\pi|$. Secondly, we have a fixed number of edges $n$, so $e=n$. Finally, we can write $\sigma$ as $\{B_1,B_3,\dots, B_{2n-1},E_1, \dots, E_r\}$, and we noticed in Section \ref{sec:anticomm} that each of the even blocks $E_j$ corresponds to a simple cycle of $\GG_\pi$, or equivalently, to an inner face of its planar representation. If we add the outer face we get that $f=r+1=|\sigma|-n+1$. Thus, Euler's formula asserts that $|\pi|-n +|\sigma|-n+1=2$ or equivalenty that $|\pi|+|\sigma|=2n+1$. But this is a well known fact of the Kreweras complement. Thus, in this case the formula satisfied by the Kreweras complement, $|\sigma|+|Kr(\sigma)|=2n+1$ is just a recast of Euler's formula. 
\end{rem}

Next we study the size of $\{ \pi\in \NN\CC(2n): \GG_\pi=\GG\}$ for a given cactus graph $\GG$. We begin by explaining why every cactus graph $\GG$ has at least one outercycle $C$ (possibly several).

\begin{rem}
\label{rem.outercycle}
Cacti graphs are outerplanar, that is, they admit a planar representation where all vertices belong to the outer (unbounded) face. In such a planar representation, each simple cycle corresponds to an inner face, and for every edge $e\in E_\GG$ there are exactly two faces that have $e$ as an edge. If $e$ is rigid, one of the faces is the simple cycle it belongs to, while the other must be the outer face. On the other hand, if $e$ is flexible, the `two' faces turn out to be the same face, which has to be the outer face. Thus, if we fix an edge and start moving around the contour of the outer face in counter-clockwise direction, we will obtain a cycle $C=(v_1,e_1,v_2,e_2,\dots,v_j,e_j)$, that passes once through every rigid edge and twice through every flexible edge, thus $C$ is an outercycle. Notice that the outercycle completely determines the planar representation of $\GG$, thus it determines $\GG$, as we can retrieve the graph by drawing the edges in order in counter-clockwise direction, being careful that we may need to return to every vertex several times. The number of outercycles of each cactus $\GG$ depends on the number of planar representations. Furthermore, for each planar representation we may have various outercycles, depending on which edge we start, and some of these outercycles may turn out to be the same, if there is an automorphism of $\GG$ sending one to another.
\end{rem}

After this discussion, we are ready to prove Proposition \ref{prop.preimage}, which asserts $|\{\pi\in \NN\CC(2n): (\GG_\pi,C_\pi)=(\GG,C)\}|=2^{f_C}$, where $f_C$ is the number flexible edges without counting the first edge $C$.

\begin{proof}[Proof of Proposition \ref{prop.preimage}]
The main idea is that if we consider a $\pi\in \NN\CC(2n)$, then we already observed that $\overrightarrow{\GG_\pi}$ is a cactus whose all simple cycles are oriented. Moreover, $C_\pi$ was constructed in a way that when a rigid edge $e_i=(v_1,v_{i+1})$ appears in $C_\pi$ it is pointing to the right $\overrightarrow{e_i}$. Thus in order to go from the pair $(\overrightarrow{\GG_\pi},C_\pi)$ to $(\GG_\pi,C_\pi)$ we just need to forget the direction of all the edges, except $\overrightarrow{e_1}$ which is oriented to the right by construction. Then, to get the preimage we have a pair $(\GG,C)$ and want to reconstruct $(\vec{\GG}_\pi,C_\pi)$, but this amounts to remembering (choosing) the direction of the flexible edges different from $e_1$, since the directions of the rigid edges and of $e_1$ are determined by $C$. And this can be done in $2^{f_C}$ different ways. The detailed proof consists of two main steps:

First, we transform the undirected graph $\GG$ into a directed graph $\overrightarrow{\GG}$. For this, we pick a direction for the edges $E_\GG$. We will do this by assigning a direction to the edges $e_1,\dots, e_j$ of $C$. For $e_r=(v_r,v_{r+1})$ we use the notation $\overrightarrow{e_i}=\overrightarrow{(v_r,v_{r+1})}$ or $\overleftarrow{e_i}=\overleftarrow{(v_r,v_{r+1})}$. We begin by assigning the simplest direction, that is $\overrightarrow{e_1},\dots, \overrightarrow{e_j}$. Since every rigid vertex appears once in $C$, this direction is well defined for these edges. Moreover, since we draw $C$ in counter-clockwise direction, this implies that every simple cycle will be also oriented in counter-clockwise direction. On the other hand, flexible edges appear twice in $C$, and actually the direction we assigned is always inconsistent. For instance if the flexible edge $e$ appears in $C$ as $e_r=(v_r,v_{r+1})$ and as $e_s=(v_s,v_{s+1})$ for $1\leq r<s\leq j$, then we must have that $v_{r+1}=v_s$ and $v_r=v_{s+1}$. This is because $e$ is not in a simple cycle, and thus if we remove $e$ from $\GG$ we disconnect the graph, (separating $v_r$ from $v_{r+1}$) and since $v_{r+1},v_{r+2},\dots, v_s$ is a path in the disconnected graph, this implies that $v_{r+1}=v_s$ and thus $v_r=v_{s+1}$. Thus we either need to assign $\overrightarrow{e_r}$ and $\overleftarrow{e_s}$ or take $\overleftarrow{e_r}$ and $\overrightarrow{e_s}$. If $e_1$ is a flexible edge, we always pick $\overrightarrow{e_1}$, since $C_\pi$ starts with the block containing element 1, and goes to the block containing element 2. For any other flexible edge, we can go for any of the two choices. After this procedure, in our sequence $\overrightarrow{e_1},\overrightarrow{e_2},\overleftarrow{e_3},\dots, \overrightarrow{e_j}$ of edges of $C$ we end up with $n$ of them pointing to the right, one for each rigid edge and one for each flexible edge. Meanwhile, the remaining $j-n$ edges of $C$, one for each flexible edge, are pointing to the left. Choosing a direction for each flexible edge different from $e_1$ can be done in two ways, so in total we have $2^{f_C}$ possible directed graphs $\overrightarrow{\GG}$ obtained from $\GG$.

The second part of the proof consists on observing that each directed oriented cactus $\overrightarrow{\GG}$ determines a unique partition $\pi\in \NN\CC(n)$ such that $(\GG_\pi,C_\pi)=(\GG,C)$. To retrieve the partition we consider in order the $n$ edges of $C$ that are pointing to the right: $\overrightarrow{e_{i_1}},\overrightarrow{e_{i_2}},\dots, \overrightarrow{e_{i_n}}$ (with $1=i_1<i_2<\dots<i_n\leq j$). Then, we set $\overrightarrow{l_r}:=\overrightarrow{e_{i_r}}$ to be the edge going from the block containing $2r-1$ to the block containing $2r$. Namely, if $\overrightarrow{l_r}=\overrightarrow{(u,v)}$ this means that $u\ni 2r-1$ and $v\ni 2r$. Notice that this uniquely determines a non-crossing partition. This is because if we have a vertex $V$ with outcoming edges $l_{k_1},\dots, l_{k_r}$ and incoming edges $l_{m_1},\dots, l_{m_s}$, this implies that the block $V$ has elements $2k_1-1, \dots, 2k_r-1$ and $2m_1,\dots,2m_s$.
\end{proof}

\begin{rem}
Notice that as a corollary of the result we just proved, we can give a new proof of Proposition \ref{prop.level.0}, which states that $|\NCacd_{2n}^{(0)}|=2^{n-1}C_n$. Recall that $\NCacd_{2n}^{(0)}$ consist of those partitions of the form $\sigma=\{B_1,B_3,\dots,B_{2n-1}\}$ that separate odd elements and do not have `even' blocks. But in terms of the graph of its Kreweras complement $\pi:=Kr(\sigma)$, this means that $\GG_{\pi}$ is a cactus graph with no cycles. This is better known as a tree. Let $OTG_n$ denote the set of oriented tree graphs, which are simply those oriented cacti graphs $(\GG,C)\in OCG_n$ such that $\GG$ is a tree. Notice that $OTG_n$ is in bijection with planar rooted trees with $n$ edges, which are counted by the Catalan number $C_n$. Indeed, given a planar rooted tree, we take $\GG$ to be the tree itself, and for the orientation $C$ we start in the root and go around in counter-clockwise direction. Moreover, in a tree all the edges are flexible, so $f_C=n-1$ (as the first edge of the cycle do not counts). Thus, we can conclude that $|\NCacd_{2n}^{(0)}|=|\{\pi\in \NN\CC(2n):(\GG_\pi,C_\pi)\in OTG_n\}=2^{n-1}C_n$.

We can also give a new proof of Proposition \ref{prop.level.max}, which asserts that
\[
|\NCacd_{4k}^{(k)}|=C_k\qquad \text{and} \qquad |\NCacd_{4k+2}^{(k)}|=(2k+1)C_{k+1},\qquad  \forall k\in \nn.
\]
In this case we look for cacti graphs $\GG$ with the largest amount of cycles, in the even case when the cactus has $2k$ edges, the largest amount of cycles can be $k$, all of which should have size 2, and all edges must be rigid. But cycles of size two are just two edges joining the same two vertices. And if we `thin' the cactus by identifying the edges in each pair, we end up with a tree on $k$ edges. Each orientation of the resulting tree gives one possible orientation of the associated cactus and we directly conclude that $|\NCacd_{4k}^{(k)}|=C_k$. The odd case can be handled in a similar way.
\end{rem}

We now proceed to prove Theorem \ref{thm.anticommutator.graphs}, which is a purely graph theoretic formula for the anti-commutator, where index now runs over bipartite cacti graphs.

\begin{proof}[Proof of Theorem \ref{thm.anticommutator.graphs}]
Recall that in Theorem \ref{Thm.anticommutator.main} we found the following formula for the anti-commutator:
\begin{equation}
\kappa_n(ab+ba) = \sum_{\substack{\pi\in \NCac_{2n} \\ \pi=\pi'\sqcup\pi''}} \Bigg( \prod_{V\in \pi'} \kappa_{|V|}(a)  \prod_{W\in \pi''} \kappa_{|W|}(b) + \prod_{V\in \pi'} \kappa_{|V|}(b)  \prod_{W\in \pi''} \kappa_{|W|}(a) \Bigg).
\end{equation}
From Proposition \ref{prop.cactus} we know that $\GG_\pi$ is always a cactus, and we already saw in Section \ref{sec:anticomm} that $\GG_\pi$ must be bipartite. In other words, all its simple cycles are of even size. So we can break the sum depending on which bipartite cactus graph $\GG_\pi$ is formed by $\pi$, moreover we consider the canonical orientation $C_\pi$. Then, $\kappa_n(ab+ba)$ is expressed as
\begin{equation}
\sum_{\substack{(\GG,C) \in OCG_n\\ \GG\text{ is bipartite}}}\  \sum_{\substack{\pi\in \NCac_{2n} \\ (\GG_\pi, C_\pi)=(\GG,C) \\ \pi=\pi'\sqcup\pi''}} \Bigg( \prod_{V\in \pi'} \kappa_{|V|}(a)  \prod_{W\in \pi''} \kappa_{|W|}(b) + \prod_{V\in \pi'} \kappa_{|V|}(b)  \prod_{W\in \pi''} \kappa_{|W|}(a) \Bigg).
\end{equation}
But we know that $(\pi',\pi'')=(V'_\GG, V''_\GG)$ is the unique bipartition of $\GG$. And we can observe that for $V\in \pi$ the size of the block $|V|$ corresponds to the degree of $V$ as a vertex of $\GG$. Thus, once we fix a graph $\GG$, no matter which partition it came from, the term of the sum is given by:
\[
\prod_{v\in V'_\GG } \kappa_{|v|}(a)  \prod_{w\in V''_\GG} \kappa_{|w|}(b) +\prod_{v\in V'_\GG }\kappa_{|v|}(b)  \prod_{w\in V''_\GG} \kappa_{|w|}(a).
\]
Moreover, from Proposition \ref{prop.preimage} we know that $|\{\pi\in \NN\CC(2n): (\GG_\pi,C_\pi)=(\GG,C)\}|=2^{f_C}.$ and this allows us to conclude that 
\[
\kappa_n(ab+ba) = \sum_{\substack{(\GG,C)\in OCG_{n}\\ \GG\text{ is bipartite}}} 2^{f_C} \Bigg( \prod_{v\in V'_\GG } \kappa_{|v|}(a)  \prod_{w\in V''_\GG} \kappa_{|w|}(b) +\prod_{v\in V'_\GG }\kappa_{|v|}(b)  \prod_{w\in V''_\GG} \kappa_{|w|}(a) \Bigg).
\]
\end{proof}

A direct application of Theorem \ref{thm.anticommutator.graphs} is Corollary \ref{cor.anticommutator.graphs.id} concerning the case where the variables have the same distribution. Another interesting application is Proposition \ref{prop.anticom.semicircular.arbitrary} which studies the case when one variable is semicircular. We present the proof of the latter as it is not straightforward as the former. 

\begin{proof}[Proof of Proposition \ref{eq.anticomm.semicircular.arbitrary}]
Recall that Equation \eqref{eq.anticommutator.graphs} asserts that for every $n\in\nn$
\[
\kappa_m(as+sa) = \sum_{\substack{(\GG,C)\in OCG_{m}\\ \GG\text{ is bipartite}}} 2^{f_C} \Bigg( \prod_{v\in V'_\GG } \kappa_{d(v)}(a)  \prod_{w\in V''_\GG} \kappa_{d(w)}(s) +\prod_{v\in V'_\GG }\kappa_{d(v)}(s)  \prod_{w\in V''_\GG} \kappa_{d(w)}(a) \Bigg).
\]
Recall that $\kappa_{d(v)}(s)=0$ unless $d(v)=2$, so the whole product will vanish unless $d(v)=2$ for all $V'_\GG$. Let us say that $V'_\GG$ and $V''_\GG$ are 2-regular if all its vertices have degree 2 in $\GG$. This means that our formula simplifies to
\[
\kappa_{m}(as+sa) = \sum_{\substack{(\GG,C)\in OCG_{m}\\ \GG\text{ is bipartite}\\ V'_\GG \text{ is 2-regular} }} 2^{f_C} \prod_{w\in V''_\GG } \kappa_{d(w)}(a)  +\sum_{\substack{(\GG,C)\in OCG_{2n}\\ \GG\text{ is bipartite}\\ V''_\GG \text{ is 2-regular} }} 2^{f_C} \prod_{v\in V'_\GG } \kappa_{d(v)}(a).
\]
Notice that since $(V'_\GG,V''_\GG)$ is a bipartition of the graph $\GG$ which has $m$ edges, this implies that $\sum_{v\in V'_\GG} d(v)=m=\sum_{v\in V''_\GG} d(v)$. But, if $V'_\GG$ or $V''_\GG$ are 2-regular this forces $m$ to be even. Thus, the product vanishes whenever $m$ is odd and we get $\kappa_{m}(as+sa)=0$. For the even case $m=2n$, let us focus on the second summation. Notice that given a bipartite graph $(\GG,C)\in OCG_{2n}$ such that $V''_\GG$ is 2-regular we can construct a graph $(\GG',C')\in OCG_{n}$, by `erasing' all the vertices $w\in V''_\GG$ but keeping track of the 2 edges ending on $w$, say $e=(u,w)\in E_\GG$ and $f=(w,v)\in E_\GG$, by merging them into the edge $g=(u,v)\in E_{\GG'}$. Formally, we let $V_{\GG'}=V'_\GG$, and draw and edge between $u,v\in V'_\GG$ if there exist a vertex $w\in V''_\GG$ such that $(u,w),(w,v)\in E_\GG$. 

Given an outercycle $C=(v_1,e_1,w_1,f_1,v_2,e_2,w_2,f_2,\dots,w_j,f_j)$ of $\GG$, where $v_i\in V'_\GG$ and $w_i\in V''_\GG$ for $i=1,\dots,j$, there is a natural way to construct an outercycle $C'=(v_1,g_1,v_2,g_2,\dots,v_j,g_j)$ of $\GG'$, by letting $g_i\in E_{\GG'}$ be the edge obtained by merging edges $e_i$ and $f_i$ of $\GG$. Notice that for each $w\in V''_\GG$, either both $(u,w)$ and $(w,v)$ are flexible edges of $\GG$ and thus $(u,v)$ is flexible edge of $\GG'$ or both are rigid edges of $\GG$ and $(u,v)$ is a rigid edge of $\GG'$. This implies that that the flexible edges of $\GG'$ are half the flexible edges of $\GG$, thus we have that if $e_1$ is flexible, then $f_{C}=2f_{C'}+1=g_{C'}$, and if $e_1$ is rigid, then $f_{C}=2f_{C'}=g_{C'}$.

The previous procedure is actually a bijection between $OCG_{n}$ and 
\[
\{(\GG,C)\in OCG_{2n}: \GG\text{ is bipartite}, V''_\GG \text{ is 2-regular} \}.
\]
This means that we have
\begin{equation}
\label{eq.semicircular.arbitrary.1}
\sum_{\substack{(\GG,C)\in OCG_{2n}\\ \GG\text{ is bipartite}\\ V''_\GG \text{ is 2-regular} }} 2^{f_C} \prod_{v\in V_{\GG'} } \kappa_{d(v)}(a)= \sum_{\substack{(\GG',C')\in OCG_{n}\\ e_1 \text{ is flexible} }} 2^{g_{C'}} \prod_{v\in V_{\GG'} } \kappa_{d(v)}(a) 
\end{equation}

With a similar, but slightly more involved procedure, we can biject $OCG_{n}$ with
\[
\{(\GG,C)\in OCG_{2n}: \GG\text{ is bipartite}, V'_\GG \text{ is 2-regular} \},
\]
and we again obtain the right-hand side of \eqref{eq.semicircular.arbitrary.1} when on the left-hand side we take $V'_\GG$ to be 2-regular instead of $V''_\GG$. Adding both cases we conclude Equation \eqref{eq.anticomm.semicircular.arbitrary}.
\end{proof}

To conclude this section let us comment on how the even case looks in this cactus graph approach. This case is particularly interesting as the graphs appearing here are the same that appear in the study of Traffic Freeness.

\begin{defi}
We say that a cactus graph $\GG$ is \emph{rigid}, if all its edges are rigid. Namely, all its edge belong to exactly one simple cycle.
\end{defi}

\begin{rem}
\label{rem.cactus.traffics}
The term `cactus' has already appeared in the study of Traffic Freeness in connection to Free Probability, see for instance \cite{male2020traffic}, \cite{au2020rigid} and the references therein. In \cite{au2020rigid}, a cactus graph has all its edges on a simple cycle, thus it is what we call rigid cactus. What we call cactus, in that paper appears under the name of quasi-cactus. Although Au and Male consider graphs with extra structure necessary for handling traffics, some of the underlying combinatorial techniques look similar. The similarities and differences between these two alike objects of study can be explained if we look at how the graphs are constructed from a non-crossing partition. In \cite{au2020rigid}, this is done via a quotient graph called $C^\pi$, which in our approach can be described as the graph $\GG_\pi$ where $n$ extra edges are drawn, each joining the blocks containing elements $2i$ and $2i+1$ for $i=1,\dots,n$. Adding these extra edges forces the cactus graph to be rigid. 
\end{rem}

The study of the anti-commutator via graphs, in the special case where the variables are even is governed by rigid cacti, as pointed out by the following result.

\begin{cor}
Consider two free even random variables $a$ and $b$. Then, for every $n\geq 1$
\begin{equation}
\label{eq.anticommutator.graphs.even}
\kappa_n(ab+ba) = \sum_{\substack{(\GG,C)\in OCG_{n}\\ \GG\text{ is bipartite}\\ \GG\text{ is rigid}}} \Bigg( \prod_{v\in V'_\GG } \kappa_{d(v)}(a)  \prod_{w\in V''_\GG} \kappa_{d(w)}(b) +\prod_{v\in V'_\GG }\kappa_{d(v)}(b)  \prod_{w\in V''_\GG} \kappa_{d(w)}(a) \Bigg).
\end{equation}
\end{cor}

\begin{proof}
We observe that cacti graphs where all its vertices have even degree must be rigid. Indeed, for the sake of contradiction assume that $\GG$ is a cactus graph where all its vertices have even degree and it has a flexible edge $e$. If we consider the graph $\GG'$ obtained by deleting $e$ from $\GG$, then $\GG'$ has two connected components, let $A$ be one of these components and let $v$ be the unique vertex of $A$ that was an endpoint of $e$. Then the degree of $v$ in $\GG'$ is odd (as $dg(v)$ in $\GG$ is even and we removed $e$), since the sum of the degrees of the vertices in $A$ is even, there must be another vertex $u$ that has odd degree, but then $d(u)$ in $\GG$ is odd, a contradiction. Notice that the converse is trivial, thus, rigid cacti are characterized as cacti where all its vertices have even degree. If we restrict the sum in \eqref{eq.anticommutator.graphs} to graphs with only even degrees, and we use previous observation we directly get \eqref{eq.anticommutator.graphs.even}.
\end{proof}

\begin{rem}
\label{rem.bijection.rigid.cactus}
Recall that in Theorem \ref{thm.anticommutator.even.variables} we already had a formula for the anti-commutator of even variables. The reason why the right-hand side of formula \eqref{eq.anticommutator.even} is the same as the right-hand side of \eqref{eq.anticommutator.graphs.even}, follows from the fact that oriented rigid cacti graphs with $2n$ edges are in bijection with $\NN\CC(n)$. This bijection should be clear from the development in \cite{au2020rigid} and previous work on Traffic Freeness related to rigid cacti graphs. Below we just give an idea on a precise bijection that uses oriented tree graphs.

Given an oriented rigid cactus graph $\GG$ consider the tree $T_\GG$ that has vertices given by the vertices and faces (simple cycles) of $\GG$, we draw an edge between a vertex $v$ and a face $f$ of $\GG$ if $v$ is contained in $f$. It is ease to see that $T_\GG$ is a tree with $n$ edges, and where the degree of $v$ in $T_\GG$ is now half the degree of $v$ in $\GG$. The outercycle $C=(v_1,e_1,v_2,\dots)$ of $\GG$ naturally gives an outercycle $C'$ of $T_\GG$ that starts in $v_1$ and whenever we had the string $(v_i,e_i,v_{i+1})$ in $C$, we replace it by the string $(v_i,e'_i,f_i,e'_{i+1},v_{i+1})$ of $C'$ that instead of going directly from $v_i$ to $v_{i+1}$ first goes to the face $f_i$ that contains $e_i$. On the other hand, non-crossing partitions are in bijection with oriented tree graphs as follows. Let $\pi\in \NN\CC(n)$ and consider the partition $\hat{\pi}=\lan \pi,Kr(\pi)\ran\in \NN\CC(2n)$, then $(\GG_{\hat{\pi}},C_{\hat{\pi}})$ is an oriented tree with $n$ edges, (it must be a tree as it has $n+1$ vertices). Moreover, in the bipartite decomposition $\hat{\pi}=\hat{\pi}' \sqcup \hat{\pi}''$ the degrees of the vertices of $\hat{\pi}'$ are exactly the sizes of the blocks of $\pi$. Putting together these bijections, gives a combinatorial explanation on why the right-hand sides of \eqref{eq.anticommutator.even} and \eqref{eq.anticommutator.graphs.even} are the same.
\end{rem}

\addtocontents{toc}{\SkipTocEntry}
\subsection{Quadratic forms}

$\ $

\noindent 
The purpose of this subsection is to prove Theorem \ref{thm.quadratic.forms}, which is a general version of the main result of the preceding subsection. This theorem enables us to compute quadratic forms on free random variables $a_1,\dots, a_k$:
\[
a:= \sum_{1\leq i,j\leq k} w_{i,j} a_ia_j.
\]

The approach is to adapt the ideas from Section \ref{sec:anticomm}, to this more general setting. In most of the cases the generalization is straightforward, so we can directly provide a proof of this result. In the proof we restrict our attention to introducing the new necessary notation and pointing out how we can modify our previous results to fit the new setting.

\begin{proof}[Proof of Theorem \ref{thm.quadratic.forms}]
Consider a $2n$-tuple $\varepsilon=(\varepsilon(1),\dots,\varepsilon(2n))\in [k]^{2n}$ with entries $\varepsilon(i)$ in $[k]:=\{1,\dots,k\}$. And for $i\in[k]$, let $A_i(\varepsilon):= \{ t\in [2n] : \varepsilon(t)=i\}$ denote the entries of $\varepsilon$ that are equal to $i$. Then, these sets together, $A(\varepsilon)=\{A_1(\varepsilon),\dots,A_k(\varepsilon)\}$, form an ordered partition of $[2n]$. For every tuple $\varepsilon\in  [k]^{2n}$, we denote its weight as 
\[
w(\varepsilon):=\prod_{r=1}^n w_{\varepsilon(2r-1),\varepsilon(2r)}.
\]

With this notation, formula \eqref{eq.rephrase} from Proposition \ref{prop.basic.anticomm} is easily generalized to 
\begin{equation}
\label{eq.inproof.1}
\kappa_n(a) =\sum_{\substack{\pi\in \NN\CC(2n)\\ \pi\lor I_{2n}=1_{2n} }} \sum_{\substack{\varepsilon\in [k]^{2n} \\ A(\varepsilon) \geq \pi  }} 
 w(\varepsilon) \prod_{i=1}^k \Bigg( \prod_{\substack{V\in \pi, \\ V\subset A_i(\varepsilon)}} \kappa_{|V|}(a) \Bigg).
\end{equation}

Then, if we fix a $\pi\in\NN\CC(2n)$ such that $\pi\lor I_{2n}=1_{2n}$. From Lemma \ref{Lemma.connected} we know that $\GG_\pi$ is connected, and proceeding as in Proposition \ref{prop.few.nonvanishing.epsilons} we further obtain that for every colored graph $\GG_\pi\in OCG^{(k)}_n$, we can construct a unique tuple $\varepsilon\in [k]^{2n}$ such that $A(\epsilon)=(Q_1,\dots, Q_k)$ (as ordered partitions). If replace this in \eqref{eq.inproof.1} we obtain a general version of our main formula
\[
\kappa_n(a) =\sum_{\substack{\pi\in \NN\CC(2n),\\ \GG_\pi\mbox{ is connected} }}  w(\varepsilon)\prod_{i=1}^k \Bigg( \prod_{\substack{V\in \pi, \\ V\subset A_i(\varepsilon)}} \kappa_{|V|}(a) \Bigg).
\]
Notice that we no longer require the graph to be bipartite, as this was a restriction coming from the fact that anti-commutator is the case $k=2$ (bicolored) with the extra requirement that $w_{1,1}=w_{2,2}=0$, which forces the graph to be bipartite. Finally, we can break the sum depending on the oriented cactus graph formed by each $\pi$. Proceeding in the same way we did for the proof of Theorem \ref{thm.anticommutator.graphs}, we conclude the desired formula \eqref{eq.quadratic.forms}.
\end{proof}

As a direct corollary of Theorem \ref{thm.quadratic.forms} we can restrict to the case where the variables $a_1,\dots, a_k$ are all even. A formula for the cumulants of a quadratic form in this case was already provided by Ejsmont and Lehner in Proposition 4.5 of \cite{ejsmont2017sample} (see also \cite{ejsmont2020sums}). The combinatorial link between their formula and the one presented below is again explained by the bijection between rigid cacti and non-crossing partitions mentioned in Remark \ref{rem.bijection.rigid.cactus}.

\begin{cor}
\label{cor.quadratic.forms.even}
Consider $k$ free even random variables $a_1,\dots, a_k$ and let $a$ be quadratic form on these variables
\[
a:= \sum_{1\leq i,j\leq k} w_{ij} a_ia_j,
\]
where $w_{ij}\in \rr$ and $w_{ji}=w_{ij}$ for $1\leq i\leq j\leq k$. Then, the cumulants of $a$ are given by 
\begin{equation}
\label{eq.quadratic.forms.even}
\kappa_n(a) = \sum_{\substack{(\GG,C) \in OCG^{(k)}_{n}\\ \GG \text{ is rigid} }} w_{\GG} \prod_{i=1}^k \Big( \prod_{v_i\in Q_i } \kappa_{|v_i|}(a_i) \Big).
\end{equation}
\end{cor}

To finish this section, we study the quadratic form given by a sum of anti-commutators. As some weights are equal to 0, this implies that the terms for some graphs vanish and the sum to the right-hand side of \eqref{eq.quadratic.forms} simplifies, and has a nice interpretation as $k$-partite graphs.

\begin{exm}
To study the sum of the anti-commutators of $k$ free random variables $a_1,\dots, a_k$
\[
a:= \sum_{1\leq i <j\leq k} a_ia_j+a_ja_i,
\]
this amounts to considering the weights $w_{ii}=0$ for $1\leq i\leq k$ and $w_{ij}=1$ for $1\leq i,j\leq k$ with $i\neq j$. Therefore, we can restrict our attention to the subset $ACSG_n^{(k)}\subset OCG^{(k)}_n$ (standing for \emph{anti-commutator sums graphs}) where the coloring $(Q_1,\dots,Q_k)$ is $k$-partition of $\GG$. Then, the cumulants of $a$ are given by 
\begin{equation}
\label{eq.sums.anticommutator}
\kappa_n(a) = \sum_{(\GG,C)\in ACSG^{(k)}_{n}} 2^{f_C}  \prod_{i=1}^k \Big( \prod_{v_i\in Q_i } \kappa_{|v_i|}(a_i) \Big).
\end{equation}

Notice that if just consider semicircular variables in this formula, then the terms in the sum will vanish for all graphs except when $\GG$ is the cycle graph on $n$ edges. This graph has only one possible orientation $C$, this cycle is the graph $\GG$ itself. Moreover $\GG$ has only rigid edges, so $f_C=0$. Thus $\kappa_n(a)$ is just the number of $k$-partitions of the cycle graph on $n$ edges. This retrieves the formulas presented in Section 6.1 of \cite{ejsmont2020free} (see also Remark 2.6 of \cite{ejsmont2020sums}).
\end{exm}

\bibliographystyle{alpha}
\bibliography{references.bib}

\end{document}